\documentclass[onefignum,onetabnum]{siamart190516}

\usepackage{stmaryrd}
\usepackage[utf8]{inputenc}
\usepackage[T1]{fontenc}
\usepackage[english]{babel}
\usepackage{textcomp}
\usepackage{lmodern}
\numberwithin{equation}{section}
\usepackage{mathrsfs}
\usepackage{enumerate}
\usepackage{enumitem}
\usepackage{dsfont}
\usepackage{color}
\usepackage{graphicx}
\usepackage{accents}
\usepackage{bm}
\usepackage{tikz, pgfplots}
\usetikzlibrary{patterns}
\usepackage[titletoc]{appendix}
\usepackage[scale = .75]{geometry}
\usepackage{lipsum}
\usepackage{amsfonts}
\usepackage{graphicx}
\usepackage{epstopdf}
\usepackage{algorithmic}
\ifpdf
  \DeclareGraphicsExtensions{.eps,.pdf,.png,.jpg}
\else
  \DeclareGraphicsExtensions{.eps}
\fi

\newsiamremark{remark}{Remark}
\newsiamremark{hypothesis}{Hypothesis}
\crefname{hypothesis}{Hypothesis}{Hypotheses}
\newsiamthm{claim}{Claim}

\usepackage{amsopn}

\newtheorem{theo}{Theorem}[section]

\newtheorem{prop}[theo]{Proposition}
\newtheorem{cor}[theo]{Corollary}
\newtheorem{lem}[theo]{Lemma}
\newtheorem{rema}[theo]{Remark}

\newtheorem{defi}[theo]{Definition}

\newcommand{\be}{\begin{equation}}
    \newcommand{\ee}{\end{equation}}
\newcommand{\bpr}{\begin{prop}}
    \newcommand{\epr}{\end{prop}}
\newcommand{\bt}{\begin{theo}}
    \newcommand{\et}{\end{theo}}
\newcommand{\bl}{\begin{lem}}
    \newcommand{\el}{\end{lem}}
\newcommand{\bc}{\begin{cor}}
    \newcommand{\ec}{\end{cor}}
\newcommand{\br}{\begin{rema}}
    \newcommand{\er}{\end{rema}}
\newcommand{\bd}{\begin{defi}}
    \newcommand{\ed}{\end{defi}}
\newcommand{\ds}{\displaystyle}

\makeatletter

\def\R{\mathbb R}

\def\R{\mathbb R}
\def\T{\mathcal T}
\def\P{\mathcal P}

\def\FKsig{{\mathcal F}_{K,\sigma}}

\def\ts{\tau_{\sigma}}

\def\DKsig{D_{K,\sigma}}

\def\E{\mathcal E}
\def\Eint{\E_{\text{int}}}
\def\EK{\E_{K}}
\def\EKint{\E_{K,int}}
\def\Eext{\E_{\text{ext}}}

\def\EextD{\E_{\text{ext}}^D}
\def\EKextD{\E_{K,ext}^D}
\def\EKextN{\E_{K,ext}^N}

\def\FKsig{{\mathcal F}_{K,\sigma}}

\def\bx{{\bm x}}
\def\bu{{\bm u}}
\def\bw{{\bm w}}

\def\bV{{\bm V}}
\def\bv{{\bm v}}

\newcommand{\dd}{\mathrm{d}}
\newcommand{\dr}{\partial}

\def\k{{{K}}}
\def\l{{{L}}}
\newcommand{\ke}{{ {K}^*}}
\def\le{{{L}^*}}
\newcommand{\sig}{ \sigma}
\newcommand{\sige}{\sigma^*}

\newcommand{\M}{\mathfrak{M}}
\newcommand{\Mie}{\mathfrak{M}^*}

\newcommand{\Me}{\Mie\cup\dr\Mie}
\newcommand{\ut}{u_{\T}}

\newcommand{\vt}{v_{\T}}

\newcommand{\n}{\boldsymbol{{\bm{n}}}}
\newcommand{\nksig}{{\bm{n}}_{\sig K}}

\newcommand{\nkesige}{\boldsymbol{{\bm{n}}}_{\sige\ke}}

\newcommand{\tkele}{\boldsymbol{{\boldsymbol{\tau}}}_{{\ke\!\!\!,\le}}}
\newcommand{\tkl}{\boldsymbol{{\boldsymbol{\tau}}}_{{K\!,L}}}

\newcommand{\xk}{\bx_{\k}}
\newcommand{\xl}{\bx_{\l}}

\newcommand{\xke}{\bx_{\ke}}
\newcommand{\xle}{\bx_{\le}}

\newcommand{\Ee}{\mathcal{E}}
\newcommand{\Ds}{{\mathcal{D}}}

\newcommand{\DD}{\mathfrak{D}}

\newcommand{\D}{{\scriptstyle\mathcal{D}}}

\newcommand{\Dsigsiget}{{\mathcal{D}_{\sigma,\sigma^*}}}

\newcommand{\disig}{{\rm d}_{\sig}}
\newcommand{\msig}{{\rm m}_{\sig}}

\newcommand{\msige}{{\rm m}_{\sige}}

\newcommand{\mk}{{\rm m}_{\k}}
\newcommand{\mke}{{\rm m}_{\ke}}
\newcommand{\md}{{\rm m}_{\Ds}}

\newcommand{\xib}{{\bm \xi}}

\newcommand{\uk}{u_{\k}}
\newcommand{\ul}{u_{\l}}
\newcommand{\uke}{u_{\ke}}
\newcommand{\ule}{u_{\le}}

\newcommand{\sumdiam}{\sum_{\Ds\in\DD}}

\newcommand{\sumpri}{\sum_{\k\in\M}}

\newcommand{\sumdua}{\sum_{\ke\in{\overline\Mie}}}

\newcommand{\gradDD}{\nabla^{\DD }}

\newcommand{\gradD}{\nabla^{\Ds}}

\makeatother

\title{Large time behavior of nonlinear finite volume schemes for convection-diffusion equations}

\author{Cl\'ement Canc\`es\thanks{Inria, Univ.~Lille, CNRS, UMR 8524 - Laboratoire Paul Painlevé, F-59000 Lille, France
  (\email{clement.cances@inria.fr}, \email{maxime.herda@inria.fr})}
\and Claire Chainais-Hillairet\thanks{Univ. Lille, CNRS, UMR 8524, Inria - Laboratoire Paul Painlevé, F-59000 Lille, France. 
  (\email{claire.chainais@univ-lille.fr})}
\and Maxime Herda\footnotemark[1]
\and Stella Krell\thanks{Universit\'e de Nice, CNRS, UMR7351-Laboratoire J.-A. Dieudonn\'e. 
        F-06100 Nice, France.
        (\email{stella.krell@unice.fr})}
}

\begin{document}

\maketitle

    \begin{abstract}
        In this contribution we analyze the large time behavior of a family of nonlinear finite volume schemes for anisotropic convection-diffusion equations set in a bounded bidimensional domain and endowed with either Dirichlet and / or no-flux boundary conditions. We show that solutions to the two-point flux approximation (TPFA) and discrete duality finite volume (DDFV) schemes under consideration converge exponentially fast toward their steady state. The analysis relies on discrete entropy estimates and discrete functional inequalities. As a biproduct of our analysis, we establish new discrete Poincaré-Wirtinger,  Beckner and logarithmic Sobolev inequalities. Our theoretical results are illustrated by numerical simulations.
    \end{abstract}

    \maketitle
    
    \smallskip
    
    
    \smallskip
    


    \section{Introduction}
    
    We are interested in the numerical discretization of linear convection diffusion equations in bounded domain with anisotropic diffusion and mixed Dirichlet-Neumann boundary conditions. More precisely, our aim is the preservation of the large time behavior of solutions at the discrete level. At the continuous level, the behavior of solutions in the large can be quantified thanks to so-called (relative) entropies. These global quantities depend on time and involve the solutions to the evolution and the stationary equations. They usually control a distance between the solution and the steady state. In dissipative models, thanks to appropriate functional inequalities of Poincaré or convex Sobolev type \cite{arnold_markowich_unterreiter_2001}, a quantitative time decay estimate of the entropy may be established. At the discrete level the challenges lie in the preservation of the dissipation of the entropy and the derivation of discrete counterparts of the functional inequalities. In the present contribution, we address both of these issues.

    Let  $T>0$ be a time horizon, $\Omega$ be a polygonal connected open bounded subset of $\R^2$ and $Q_T = \Omega\times (0,T)$. The boundary $\Gamma = \partial\Omega$ is divided in two parts $\Gamma = \Gamma^D\cup\Gamma^N$, denoted $\Gamma^D$ and $\Gamma^N$ which will be endowed with respectively non-homogeneous Dirichlet and no-flux boundary conditions. In the following we are interested in the numerical approximation of the solution $u\equiv u(\bm{x},t)$ of 
    \begin{subequations}\label{FPmodel}
        \begin{align}
            &\ds\partial_t u\,+\,{ \rm div} {\bm J}\ =\ 0 \quad\mbox{in}\quad Q_T\,,\label{FPmodel:cons}\\
            &\ds {\bm J}\ =\ -{\bm \varLambda}(\nabla u \,+\, u  \nabla  V)\quad\mbox{in}\quad Q_T\,,\label{FPmodel:flux}\\
            &\ds{\bm J}\cdot {\bm n}\ =\ 0\quad\mbox{on}\quad \Gamma^N\times (0,T)\,,\label{FPmodel:CLneum}\\
            &\ds u\ =\ u^D\quad\mbox{in}\quad \Gamma^D\times (0,T)\,,\label{FPmodel:CLdir}\\
            &\ds u(\cdot,0)\ =\ u_0\quad\mbox{in}\quad\Omega\,,\label{FPmodel:CI}
        \end{align}
    \end{subequations}
    where ${\bm n} $ denotes the outward unit normal to $ \partial \Omega$. We make the following assumptions on the data.
    \begin{enumerate}\label{hyp}
        \item[(A1)] The initial data $u_0$ is square-integrable and
        non-negative, i.e., $u_0\in L^2(\Omega)$ and $u_0 \geq 0$. 
        In the pure Neumann case $\Gamma^D = \emptyset$, we further
        assume that the initial data is non-trivial, i.e. 
        \be\label{eq:M1}
        M_1 = \int_{\Omega} u_0 {\rm d}{\bx} >0.
        \ee
        \item[(A2)] The exterior potential $V$ does not depend on time
        and belongs to $C^1({\overline\Omega},\R)$. 
        \item[(A3)] If $\Gamma^D \neq \emptyset$, the boundary data $u^D$ corresponds to a thermal Gibbs equilibrium, i.e. 
        \be\label{eq:thermal_eq_data}
        u^D(\bx) = \rho e^{-V(\bx)}, \qquad \forall \bx \in \Gamma^D
        \ee
        for some $\rho>0$. As a consequence, $\log u^D + V$ is constant on $\Gamma^D$.
        \item[(A4)] The anisotropy tensor ${\bm \varLambda}$ is supposed
        to be bounded (${\bm \varLambda} \in
        L^{\infty}(\Omega)^{2\times2}$), symmetric and uniformly
        elliptic. There exists $\lambda^M\geq\lambda_m>0$ such that 
        \be\label{ellipticity}
        \lambda_m\vert {\bm v}\vert^2\leq {\bm \varLambda}(\bx){\bm v}\cdot {\bm v}\leq \lambda^M\vert {\bm v}\vert^2 \quad \mbox{ for all ${\bm v}\in\R^2$ and almost all $\bx\in \Omega$.}
        \ee
    \end{enumerate}
    
    The large time behavior of solutions to Fokker-Planck equations with isotropic diffusion (namely \eqref{FPmodel:cons}-\eqref{FPmodel:flux} with ${\bm \varLambda}={\bm I}$) has been widely studied by Carrillo, Toscani and collaborators, see \cite{carrillo_toscani_1998,Toscani_99,carrillo_etal_2001} thanks to relative entropy techniques. In these papers, the exponential decay of the entropy is established  in the whole space $\Omega=\R^d$ or for special cases of boundary conditions ensuring that the steady-state $u^{\infty}$ is a Gibbs equilibrium (or thermal equilibrium), which means $u^{\infty}=\lambda e^V$, with $\lambda\in \R_+$. The case of more general Dirichlet boundary conditions and anisotropic diffusion, leading to different steady states has been recently dealt with by Bodineau {\em et al.} in \cite{bodineau_2014_lyapunov}. The method is still based on relative entropy techniques.
    
    When designing numerical schemes for the convection diffusion equations \eqref{FPmodel}, it is crucial to ensure that the scheme has a similar large time behavior as the continuous model. Indeed, it ensures the reliability of the numerical approximation in the large. Also it upgrades  local-in-time quantitative convergence estimates to global-in-time estimates (see Li and Liu \cite{li_2018_large}). 
    
    {The preservation of the long-time behavior at the discrete level starts with a structure-preserving design of the numerical scheme. This has been widely investigated for Fokker-Planck type models (see, non-exhaustively, \cite{scharfetter_1969_large, ilin_1969_difference, chang1970practical, larsen_1985_discretization, buet_2010_chang, liu_2012_entropy,cances_guichard_2017,CCHK, bailo2018fully, bianchini_2018_truly, pareschi2018structure, gosse_2020_aliasing}). However the sole preservation of a stationary state, entropy inequality or well-balanced structure is not sufficient  to rigorously derive  explicit rates of convergence to equilibrium. Recently there has been an effort made on the obtention quantitative estimates with explicit decay rates by the mean of discrete functional inequalities  \cite{bessemoulin_chainais_2017,filbet_herda_2017,chainais_herda, li_2018_large,dujardin2018coercivity,bessemoulin2018hypocoercivity}. In this direction, the case of multi-dimensional anisotropic diffusion and general meshes has essentially not been dealt with yet, which leads us to the present contribution.}

    In \cite{chainais_herda}, Chainais-Hillairet and Herda prove that a family of TPFA (Two-Point Flux Approximation) finite volume schemes for \eqref{FPmodel} with ${\bm \varLambda}={\bm I}$ satisfies the exponential decay towards the associated discrete equilibrium. This family of B-schemes includes the classical centered scheme, upwind scheme and Scharfetter-Gummel scheme (\cite{scharfetter_1969_large}). Let us mention that the Scharfetter-Gummel scheme is the only B-scheme of the family that preserves Gibbs/thermal equilibrium. Unfortunately, the B-schemes are based on a two-point flux approximation and they can only be used on restricted meshes. In order to deal with almost general meshes and with anisotropic tensors, Canc\`es and Guichard propose and analyze a nonlinear VAG scheme for the approximation of some generalizations of \eqref{FPmodel} in \cite{cances_guichard_2017}. In \cite{CCHK}, Canc\`es, Chainais-Hillairet and Krell establish the convergence of a free-energy diminishing discrete duality finite volume (DDFV) scheme for \eqref{FPmodel} with $\Gamma^D=\emptyset$. Some numerical experiments show the exponential decay of the numerical scheme towards the Gibbs/thermal equilibrium. In the present contribution, we establish this result theoretically. 
    
    {In order to prove our main results on the large-time behavior of nonlinear finite volume schemes, namely Theorem~\ref{th:expdec_TPFA_N} and Theorem~\ref{th:expdec_TPFA_DN} for TPFA schemes and Theorem~\ref{theo:expdec_DDFV} for DDFV schemes, we rely upon new discrete Poincaré-Wirtinger, Beckner and logarithmic-Sobolev inequalities that are established in Theorem~\ref{theo:func_ineq}. For previously existing results on discrete adaptation of functional inequalities (Poincaré, Poincaré-Wirtinger, Poincaré-Sobolev and  Gagliardo-Nirenberg-Sobolev) for finite volume schemes we refer to the work by Bessemoulin-Chatard {\em et al.} \cite{bessemoulin_chainais_filbet_2015} and references therein. Concerning convex Sobolev inequalities (Beckner and log-Sobolev), we refer to Chainais-Hillairet {\em et al.} \cite{chainais_jungel_schuchnigg_2016} and Bessemoulin-Chatard and J\"ungel \cite{MBC_AJ_2014}. In the previous papers the reference measure in the inequality, which is related to the steady state of a corresponding convection-diffusion equation, is uniform. Recently there were occurrences of new discrete functional inequalities of Poincaré-Wirtinger type associated with discretizations of nontrivial reference measures (essentially Gaussian), see Dujardin {\em et al.} \cite{dujardin2018coercivity}, Bessemoulin-Chatard {\em et al.} \cite{bessemoulin2018hypocoercivity} and Li and Liu \cite{li_2018_large}. In the present contribution, we deal with the case of discretizations of any absolutely continuous positive measure (in the bounded domain $\Omega$) with density bounded from above and away from~$0$.}
    
    \medskip
    
    \noindent{\bf Outline of the paper.} As a first step, we focus on nonlinear TPFA finite volume schemes for \eqref{FPmodel} in the isotropic case. These schemes can be seen as the reduction of the nonlinear DDFV scheme of \cite{CCHK} to some specific meshes. In Section~\ref{sec:TPFA}, we present the schemes and we establish a discrete entropy/dissipation property and some \emph{a priori} estimates satisfied by a solution to the scheme. They permit to establish the existence of a solution to the scheme.  Then, Section \ref{sec:TPFA_LTB} is devoted to the study of the long time behavior of the nonlinear TPFA schemes: we establish the exponential decay towards equilibrium. 
    In Section~\ref{sec:DDFV}, we consider the nonlinear DDFV scheme introduced in \cite{CCHK} for the anisotropic case and almost general meshes and we also establish its exponential decay towards equilibrium. Then, in Section~\ref{sec:logSob}, we prove the various functional inequalities which are crucial in the proof of the exponential decay in the case of no-flux boundary conditions. Finally, Section~\ref{sec:num_exp} is dedicated to numerical experiments.

    \section{The nonlinear two-point flux approximation (TPFA) finite volume schemes}\label{sec:TPFA}
    
    In this section, we introduce a particular family of finite volume schemes for \eqref{FPmodel} in the isotropic case ${\bm \varLambda}={\bm I}$ . They are based on a nonlinear two-point discretization of the following reformulation of the flux
    \[
    {\bm J}\ =\ -u\,\nabla (\log u +V)\,.
    \]
    In the following, we start by presenting the schemes. Then, we establish some \emph{a priori} estimates, which finally lead to the existence of a solution to the scheme.

    \subsection{Presentation of the schemes}
    
    Let us first introduce the notations describing the mesh. The mesh $\M=(\T,\E,\P)$ of the domain $\Omega$ is given by a family $\T$ of open polygonal control volumes, a family $\E$ of edges, and a family $\P=(\bx_{K})_{K\in\T}$ of points such that  $\bx_K\in K$ for all $K\in\T$. As it is classical for TPFA  finite volume discretizations including diffusive terms, we assume that the mesh is admissible in the sense of \cite[Definition 9.1]{EGHbook}. It implies that the straight line between two neighboring centers of cells $(\bx_{K},\bx_{L})$ is orthogonal to the edge $\sigma=K|L$.\\
    In the set of edges $\E$, we distinguish the interior edges $\sigma=K|L\in \Eint$ and the boundary edges $\sigma\in\Eext$. Within the exterior edges, we distinguish the Dirichlet boundary edges included in $\Gamma^{D}$ from the Neumann (no-flux) boundary edges included in $\Gamma^{N}$: $\Eext=\Eext^{D}\cup\Eext^{N}$. For a control volume $K\in\T$, we define $\EK$ the set of its edges, which is also split into $\EK=\EKint\cup\EKextD\cup\EKextN$. For each edge $\sigma\in\E$, there exists at least one cell $K\in\T$ such that $\sigma\in\EK$. 
    Moreover, we define $\bx_{\sigma}$ as the center of $\sigma$ for all $\sigma\in\E$.
    
    In the sequel, $d(\cdot,\cdot)$ denotes the {Euclidean} distance in $\R^2$ and $m(\cdot)$ denotes {both the Lebesgue measure and the $1$-dimensional Hausdorff
measure in $\R^2$}.  For all $K\in\T$ and $\sigma\in\E$, we set $\mk=m(K)$ and $\msig=m(\sigma)$. For all $\sigma\in\E$, we define $\disig=d(\bx_{K},\bx_{L})$ if $\sigma=K|L\in\Eint$ and $\disig=d(\bx_{K},\sigma)$, {which denotes the Euclidean distance between $x_K$ and its orthogonal projection on the line containing $\sigma$,} if $\sigma\in\Eext$, with $\sigma\in\EK$. Then the transmissibility coefficient is defined by $\ts=\msig/\disig$, for all $\sigma\in\E$.
    We assume that the mesh satisfies the following regularity constraint:
    \begin{equation}\label{regmesh}
        \text{There exists }\zeta >0\text{ such that }d(\bx_{K},\sigma)\geq \zeta\,\disig,\ \text{for all }K\in\T\text{ and }\sigma\in\EK\,.
    \end{equation}
    Let $\Delta t>0$ be the time step. We define $N_{T}$ the  integer  part of $T/\Delta t$ and $t^{n}=n\Delta t$ for all $0\leq n\leq N_{T}$. The size of the mesh is defined by $\text{size}(\T)=\max_{K\in\T}\text{diam}(K)$, {where $\text{diam}(K) = \sup_{x,y\in K}d(x,y)$} and we denote by $\delta=\max(\Delta t,\text{size}(\T))$ the size of the space--time discretization.
    
    A finite volume scheme for a conservation law with unknown $u$ provides a vector $\bu_{\T}=(u_{K})_{K\in\T}\in\R^{\theta}$ (with $\theta=\text{Card}(\T)$) of approximate values.
    However, since there are Dirichlet conditions on a part of the boundary, we also need to define approximate values for $u$ at the corresponding boundary edges: $\bu_{\E^{D}}=(u_{\sigma})_{\sigma\in\EextD}\in\R^{\theta^{D}}$ (with $\theta^{D}=\text{Card}(\EextD)$). Therefore, the vector containing the approximate values both in the control volumes and at the Dirichlet boundary edges is denoted by $\bu=(\bu_{\T},\bu_{\E^{D}})$.

    For all $K\in\T$ and $\sigma\in\E_{\text{ext}}^D$, we introduce  
    $V_K=V(\bx_K)$ and $V_{\sigma}= V(\bx_\sigma)$,
    and the associated $\bV=(\bV_\T,\bV_{\E^D})$. The boundary data $u^D$ is discretized by $u_\sigma=u^D(\bx_\sigma)$ for all $\sigma \in \E_{\text{ext}}^D$, so that $u_\sigma = \rho e^{-V_\sigma}$ according to (A3).

    For any vector $\bu=(\bu_{\T},\bu_{\E^{D}})$, we define the neighbor unknown for all $K\in\T$ and all $\sigma\in\EK$ to be
    \begin{equation}\label{uKsig}
        u_{K,\sigma}=\left\{\begin{array}{ll}
            u_{L}& \text{ if } \sigma=K|L\in\EKint,\\
            u_{\sigma}& \text{ if } \sigma\in\EKextD,\\
            u_{K} & \text{ if }\sigma\in\EKextN.
        \end{array}\right.
    \end{equation}
    We also define the difference operators, for all $K\in\T$ and $\sigma\in\E_K$ by 
    \begin{equation*}
        D_{K,\sigma}\bu=u_{K,\sigma}-u_{K}\,, \qquad D_{\sigma}\bu=|D_{K,\sigma}u|\,.
    \end{equation*}
    {Observe that $D_{\sigma}\bu$ does not depend on the control volume $K$ but only on the edge $\sigma$. Indeed, if $\sigma = K|L$ then $|D_{K,\sigma}\bu|= |D_{L,\sigma}\bu|$, if $\sigma\in\E_{\text{ext}}^N$ then $|D_{K,\sigma}\bu|=0$ and if $\sigma\in\E_{\text{ext}}^D$ there is a unique $K$ such that $\sigma\in\E_K$.}

    The family of schemes we consider in this section writes as:
    \begin{subequations}\label{scheme}
        \begin{align}
            &\mk\,\frac{u_K^{n+1}-u_K^n}{\Delta t} \,+\,\sum_{\sigma\in \E_K} \FKsig^{n+1}\ =\ 0\quad\text{for all }K\in\T\text{ and }n\geq 0\,,\label{scheme_glob}\\[.5em]
            & \FKsig^{n+1}\ =\ -\ts {\overline u}_{\sigma}^{n+1}\,\DKsig(\log \bu^{n+1} \,+\, \bV) \quad \text{for all }K\in\T\,,\ \sigma \in\E_K\text{ and }n\geq 0\,,\label{scheme_flux}\\[.5em]
            &u_\sigma^{n+1}\ =\ u_\sigma^D\quad\text{for all }\sigma \in \E_{\text{ext}}^D\text{ and }n\geq 0,\label{scheme_CLdir}\\[.5em]
            &u_K^0=\frac{1}{\mk }\int_K u_0(\bx) \,\dd\bx\quad\text{for all } K\in\T.\label{scheme_CI}
        \end{align}
    \end{subequations}
    Let us first remark that the definition \eqref{uKsig} ensures that the numerical fluxes $\FKsig^{n+1}$ defined by \eqref{scheme_flux} vanish on the Neumann boundary edges. 
    It remains to define the values ${\overline u}_{\sigma}^{n+1}$ for the interior edges and the Dirichlet boundary edges. We define ${\overline u}_{\sigma}^{n+1}$ as a ``mean value'' of $u_K^{n+1}$ and $u_L^{n+1}$ if $\sigma=K|L$ or a mean value of $u_K^{n+1}$ and $u_\sigma^D$ if $\sigma\in\E_K^D$. More precisely, we set
    \begin{equation}\label{def:ubarsig}
        {\overline u}_{\sigma}^{n+1}=\left\{\begin{array}{ll}
            r(u_K^{n+1},u_L^{n+1})& \text{ if } \sigma=K|L\in\EKint,\\[2.mm]
            r(u_K^{n+1},u_\sigma^{D})& \text{ if } \sigma\in\EKextD,
        \end{array}
        \right.
    \end{equation}
    where $r:(0,+\infty)^2\to (0,+\infty)$ satisfies the following properties.
    \begin{subequations}\label{hyp:g}
        \begin{align}
            & \mbox{$r$ is monotonically increasing with respect to  both variables;}\label{hyp:g:monotony}\\
            & \mbox{$r(x,x)=x$ for all $x\in (0,+\infty)$ and $r(x,y)=r(y,x)$ for all  $(x,y)\in (0,+\infty)^2$;}\label{hyp:g:cons+cons}\\
            &\mbox{$r(\lambda x,\lambda y)=\lambda r(x,y)$  for all $\lambda >0$ and all $(x,y)\in (0,+\infty)^2$;}\label{hyp:g:homogeneity}\\
            &\mbox{$\displaystyle\frac{x-y}{\log x-\log y} \leq r(x,y)\leq \max(x,y)$ for all $(x,y)\in (0,+\infty)^2$, $x\neq y$.\label{hyp:g:bounds}}
        \end{align}
    \end{subequations}
    Let us emphasize that one has for all $x,y>0$
    \begin{equation}\label{eq:ineg_r}
    \frac{x-y}{\log x-\log y}\ \leq\ \left(\frac{\sqrt{x}+\sqrt{y}}{2}\right)^2\ \leq\ \frac{x+y}{2}\leq \max(x,y)\,,
\end{equation}
    and that each function appearing in the last sequence of inequalities satisfies the properties \eqref{hyp:g}.
    
    \subsection{Steady state of the scheme}
    We say that $\bu^\infty=(\bu_{\T}^\infty,\bu_{\E^{D}}^\infty)$ is a steady state of the scheme \eqref{scheme} if it satisfies 
    \begin{equation}\label{scheme_stat_glob}
        \sum_{\sigma\in \E_K} \FKsig^{\infty}\ =\ 0\quad\text{for all }K\in\T\,,
    \end{equation}
    with the steady flux defined for all $K\in\T$ and $\sigma \in\E_K$ as 
    \begin{equation}\label{scheme_stat_flux}
        \FKsig^{\infty}\ =\ -\ts {\overline u}_{\sigma}^{\infty}\,\DKsig(\log \bu^{\infty} \,+\, \bV)\,,
    \end{equation}
    as well as the boundary/compatibility conditions 
    \begin{equation}\
        \left\{
        \begin{array}{ll}\label{stat_boundary_compat}
            \ds u_\sigma^{\infty}\ =\ u_\sigma^D\quad\text{for all }\sigma \in \E_{\text{ext}}^D\,,&\text{ if }\E_{\text{ext}}^D\neq\emptyset\,,\\[.5em]
            \ds\sum_{K\in\T} \mk  u_K^\infty=\ds\sum_{K\in\T} \mk  u_K^0=\int_\Omega u_0\ =:\ M^1\,,&\text{ if }\E_{\text{ext}}^D=\emptyset\,.
        \end{array}
        \right.
    \end{equation}
    
    In the case $\E_{\text{ext}}^D=\emptyset$, namely with full no-flux boundary conditions, the condition \eqref{stat_boundary_compat} is imposed to ensure uniqueness of the steady state and compatibility with the conservation of mass which is satisfied by the scheme. Indeed, one has
    \begin{equation}\label{eq:masscons}
        \ds\sum_{K\in\T} \mk  u_K^n=\ds\sum_{K\in\T} \mk  u_K^0=M^1,\quad \forall n\geq 0.
    \end{equation}
    The steady state associated to the scheme \eqref{scheme} 
    is given by 
    \be\label{Neumann:steadystate}
    u_K^{\infty}\ =\ \rho e^{-V_K}\,,\ \mbox{ with }\ \rho\ =\ M^1\,\left(\ds\sum_{K\in\T} \mk  e^{-V_K}\right)^{-1}.
    \ee
    
    In the case $\E_{\text{ext}}^D\neq\emptyset$, Assumption (A3) enforces the boundary conditions to be at thermal equilibrium, which means that there is a constant $\alpha^D$ such that for all $\sigma\in\E_{\text{ext}}^D$,
    \be\label{hyp:theq}
    \log u_\sigma^D \,+\,V_\sigma\ =\ \alpha^D\,.
    \ee
    Under this assumption, the steady-state has a similar form as in the case of pure no-flux boundary conditions. It is defined by 
    \be\label{theq:steadystate}
    u_K^{\infty}\ =\ \rho\, e^{-V_K}\,,\  \mbox{ with }\ \rho\ =\ \exp {\alpha^D}.
    \ee
    Let us remark that in both cases, as $V\in C^1(\overline{\Omega},\R)$, the discrete steady state is bounded by above and below. There are $M^{\infty}, m^{\infty}>0$ such that for all $K\in \T$
    \be\label{bound:steadystate}
    m^{\infty}\,\leq\, u_K^{\infty}\,\leq\, M^{\infty}\,.
    \ee

    \subsection{Discrete entropy estimates}
    Let $\Phi\in C^1(\R,\R)$ be a convex function satisfying $\Phi(1)=\Phi'(1)=0$. 
    We consider the discrete relative $\Phi$-entropy defined by 
    \be\label{def:entropy}
    {\mathbb E}_{\Phi}^n\ =\ \ds\sum_{K\in\T}\mk  u_K^{\infty} \Phi\left(\ds\frac{u_K^{n}}{u_K^{\infty}}\right)\,\quad \forall n\geq 0.
    \ee
    We show in the next proposition that if the discrete equilibrium is a Gibbs/thermal equilibrium, the scheme dissipates the discrete relative $\Phi$-entropies along time.
    \begin{prop}\label{prop:entropydissipation}
        Let us assume that either $\E_{\text{ext}}^D=\emptyset$ or $\E_{\text{ext}}^D\neq\emptyset$ with \eqref{hyp:theq}. We also assume that the scheme \eqref{scheme}-\eqref{def:ubarsig} has a solution $(\bu^n)_{n\geq 0}$ which is positive at each time step, namely $u_K^n>0$ for all $K\in\T$ and $n\geq 0$. Then the discrete relative $\Phi$-entropies defined by \eqref{def:entropy} are dissipated along time. Namely, for all $n\geq 0$ one has
        \be\label{ineq:entropydissipation}
        \ds\frac{{\mathbb E}_{\Phi}^{n+1}-{\mathbb E}_{\Phi}^n}{\Delta t} \,+\,{\mathbb I}_\Phi^{n+1} \ \leq\ 0\,,
        \ee
        with 
        \be\label{def:dissipation}
        {\mathbb I}_{\Phi}^{n+1}\,=\,\sum_{\sigma\in\E_{\text{int}}\cup \E_{\text{ext}}^D} \ts {\overline u}_{\sigma}^{n+1}\left(D_{K,\sigma} \log \frac{\bu^{n+1}}{\bu^{\infty}}\right)\left(D_{K,\sigma} \Phi'\left(\frac{\bu^{n+1}}{\bu^{\infty}}\right)\right)\ \geq\ 0.
        \ee
    \end{prop}
    
    \begin{proof}
        Regardless of the hypothesis on $\E_{\text{ext}}^D$, the steady-state can be written as $u_K^{\infty}=\rho e^{-V_K}$ with $\rho\in (0,+\infty)$, as shown in \eqref{Neumann:steadystate} and \eqref{theq:steadystate}. Therefore, the numerical fluxes defined by \eqref{scheme_flux} rewrite as
        \[
        \FKsig^{n+1}\ =\ -\ts {\overline u}_{\sigma}^{n+1} \DKsig\left(\log (\bu^{n+1}/\bu^\infty)\right)\,,
        \]
        for all $K\in\T$,  $\sigma \in \E_K$ and $n\geq 0$. Besides, due to the convexity of $\Phi$, one has
        \[
        \ds{\mathbb E}_{\Phi}^{n+1}-{\mathbb E}_{\Phi}^n\ \leq\ \ds\sum_{K\in\T} \mk  (u_K^{n+1}-u_K^n) \Phi'(u_K^{n+1}/u_K^{\infty})\,.
        \]
        Then, multiplying the scheme \eqref{scheme_glob} by $\Phi'(u_K^{n+1}/u_K^{\infty})$, summing over $K\in\T$ and applying  a discrete integration by parts yields the expected result. Finally, from the monotonicity of the functions $\log$ and $\Phi'$ we infer that ${\mathbb I}_{\Phi}^{n+1}$ is non-negative.
    \end{proof}
    
    The first consequence of Proposition \ref{prop:entropydissipation} is the decay of the relative $\Phi$-entropy, so that 
    \begin{equation}\label{entropybound}
        {\mathbb E}_{\Phi}^n\ \leq\ {\mathbb E}_{\Phi}^0, \quad \text{for all}\ n\geq 0\,. 
    \end{equation}
    Then, we deduce some uniform $L^{\infty}$-bounds on the solution to the scheme \eqref{scheme}.
    
    \begin{prop}\label{prop:Linfbounds}
        Under the assumptions of Proposition~\ref{prop:entropydissipation}, one has
        \be\label{est:Linfbounds}
        m^{\infty}\min \left(1,\ds\min_{K\in\T} \frac{u_K^0}{u_K^{\infty}}\right)\leq u_K^n\leq M^{\infty}\max \left(1,\ds\max_{K\in\T} \frac{u_K^0}{u_K^{\infty}}\right)\,,
        \ee
        for all $K\in\T$ and $n\geq 0$.
    \end{prop}
    \begin{proof}
        The proof is similar to the proof of Lemma 4.1 in \cite{filbet_herda_2017}. It is a direct consequence of \eqref{entropybound} applied with specific choices for the function $\Phi$. Indeed, just use $\Phi(x)=(x-M)^+$ and $\Phi(x)=(x-m)^-$ with  $M=\max(1,\ds\mathrm{max}_{K} u_K^0/u_K^{\infty})$ and $m=\min(1,\ds\mathrm{min}_{K} u_K^0/u_K^{\infty})$ so that in both cases $0\leq{\mathbb E}_{\Phi}^n\leq{\mathbb E}_{\Phi}^0=0$ for all $n\geq 0$, which leads respectively to the upper bound and the lower bound in \eqref{est:Linfbounds}.
    \end{proof}

    \subsection{Existence of a solution to the scheme}
    
    The numerical scheme \eqref{scheme}-\eqref{def:ubarsig} amounts at each time step to solve a nonlinear system of equations. The existence of a solution to the scheme is stated in Proposition~\ref{theo:existence}. It is a direct consequence of the \emph{a priori} $L^{\infty}$-estimates given in Proposition \ref{prop:Linfbounds}. The proof relies on a topological degree argument/a Leray-Schauder's fixed point theorem \cite{leray_schauder_1934, deimling_1985_nonlinear, droniou_2018_gradient}. It will be omitted here.
    
    \begin{prop}\label{theo:existence}
        Let us assume that either $\E_{\text{ext}}^D=\emptyset$ or $\E_{\text{ext}}^D\neq\emptyset$ with \eqref{hyp:theq}. We also assume that the initial condition satisfies \eqref{hyp:ci}. Then, the scheme \eqref{scheme}-\eqref{def:ubarsig} has a solution $(\bu^n)$ for all $n\geq 0$, which satisfies the uniform $L^{\infty}$-bounds \eqref{est:Linfbounds}. 
    \end{prop}
    
    \br
    The lower bound in \eqref{est:Linfbounds} is positive if there is a positive constant $m_0$ such that
    \begin{equation}\label{hyp:ci}
        u_K^0\,\geq\, m_0\,>\,0\,,\ \mbox{for all}\ K\in\T,
    \end{equation}
    which is not necessarily ensured by the assumption (A1) on the initial data $u_0$. If the initial data $u_0$ has some vanishing zones, such that 
    \eqref{hyp:ci} is not satisfied, it is still possible to obtain a positive lower bound at each time step $n\geq 1$. But this bound will depend on the discretization parameters. Instead of the bound of the entropy, the proof uses the control on the dissipation of entropy also provided by Proposition \ref{prop:entropydissipation}.  We refer to \cite[Lemma 3.5]{CCHK} for the details of the proof in this case. 
    This weaker estimate is sufficient to show existence of a solution $\bu^n$ to the scheme.
    \er
    
    \section{Large time behavior of the nonlinear TPFA finite volume schemes}\label{sec:TPFA_LTB}
    
    In this section, we establish the exponential decay of the solution $(\bu^n)_{n\geq 0}$ to the scheme \eqref{scheme}-\eqref{def:ubarsig}, discretizing  \eqref{FPmodel} in the particular case $\bm{\varLambda} = \bm{I}$, towards the thermal equilibrium ${\bu}^\infty$. To proceed, we first prove the exponential decay of some relative entropies ${\mathbb E}_\Phi^n$ towards $0$. We shall focus on the Boltzmann-Gibbs entropy generated by
    \begin{equation}\label{eq:phi1}
        \Phi_1(s)\ =\ s\log s -s +1\,,
    \end{equation}
    and the Tsallis entropies generated by
    \begin{equation}\label{eq:phip}
        \Phi_p(s)\ =\ \frac{s^p-ps}{p-1}+1\,,
    \end{equation}
    for $p\in(1,2]$. The methodology consists in establishing a so-called entropy-entropy dissipation inequality. More precisely, one wants to show the existence of some $\nu>0$ such that 
    \be\label{ineq:EI}
    {\mathbb I}_\Phi^{n+1}\ \geq\ \nu {\mathbb E}_\Phi^{n+1},\quad \forall n\geq 0.
    \ee
    This is done thanks to discrete functional inequalities. In the case of complete Neumann (no-flux) boundary conditions we need new inequalities that are proved in Section~\ref{sec:logSob}. Depending on the parameter $p$, we will use a logarithmic Sobolev inequality ($p=1$), a Beckner inequality ($p\in(1,2)$) or a Poincaré-Wirtinger inequality ($p=2$). In the case of mixed Dirichlet-Neumann boundary conditions, we require a more classical discrete Poincaré inequality.

    Once one obtains \eqref{ineq:EI}, due to the entropy/entropy dissipation inequality \eqref{ineq:entropydissipation}, we get that 
    $
    {\mathbb E}_\Phi^{n}\leq(1+\nu \Delta t)^{-n} {\mathbb E}_\Phi^{0}$ for all $n\geq0$. 
    Thus we deduce the following weaker but maybe more explicit bound. For any $k>0$, if $\Delta t\leq k$, then 
    $
    {\mathbb E}_\Phi^{n}\ \leq\ e^{-{\tilde\nu} t^n}{\mathbb E}_\Phi^{0}
    $
    where the rate is given by $\tilde\nu=\log (1+\nu k)/k$.
    
    \subsection{The case of Neumann boundary conditions}
    
    In this section we show the exponential decay towards the thermal equilibrium in the case of Neumann (no-flux) boundary conditions.
    
    \bt\label{th:expdec_TPFA_N}
    Let us assume that $\E_{\text{ext}}^D=\emptyset$. Then for all $p\in[1,2]$, there exists $\nu_p$ depending only on the domain $\Omega$, the regularity of the mesh $\zeta$, the mass of the initial condition $u_0$ (only in the case $p=1$) and the potential $V$ ({\em via} the steady state ${\bu}^\infty$), such that, 
    \be\label{expdec_Ep}
    {\mathbb E}_p^{n}\ \leq\ (1\,+\,\nu_p \Delta t)^{-n}\, {\mathbb E}_p^{0}\,,\quad \forall n\geq 0\,.
    \ee
    Thus for any $k>0$, if $\Delta t\leq k$, one has for all $n\geq0$ that ${\mathbb E}_p^{n}\leq e^{-{\tilde \nu_p} t^n}{\mathbb E}_p^{0}$ with ${\tilde\nu}_p=\log (1+\nu_p k)/k$. 
    \et
    
    \begin{proof}
        By definition \eqref{def:dissipation}, the discrete entropy dissipation is given by
        \[
        {\mathbb I}_p^{n+1}=\ds\sum_{\sigma\in\E_{\text{int}}}\tau_\sigma {\overline u}_{\sigma}^{n+1} \left(D_\sigma \log(\bu^{n+1}/\bu^{\infty})\right)\left(D_\sigma \Phi_p'(\bu^{n+1}/\bu^{\infty})\right)\,,
        \]
        for all $n\geq0$. It can be seen as the discrete counterpart of  
        \[
        \int_\Omega u\nabla \log (u/u^{\infty})\cdot\nabla \Phi_p'(u/u^{\infty}) \dd\bx = \frac 4 p \int_\Omega u^{\infty}|\nabla (u/u^{\infty})^{p/2}|^2 \dd\bx.
        \]
        Let us introduce a discrete counterpart of this last quantity. For all $n\geq 0$, let
        \[
        {\widehat {\mathbb I}}_p^{n+1}\,=\,\frac{4}{p}\ds\sum_{\sigma\in\E_{\text{int}}}\tau_\sigma {\overline u}_{\sigma}^{\infty} \left(D_\sigma (\bu^{n+1}/\bu^{\infty})^{p/2}\right)^2\,,
        \]
        with
        \[
        {\overline u}_{\sigma}^{\infty}\,=\,\min (u_K^{\infty},u_L^{\infty})\ \mbox{for}\ \sigma=K|L\,.
        \]
        Let us prove now that 
        \be\label{ineg:I_Ihat}
        {\widehat {\mathbb I}}_p^{n+1}\leq {{\mathbb I}}_p^{n+1},\quad  \forall n\geq 0.
        \ee
        The proof is based on two elementary inequalities. Let $x,y>0$. The first inequality is 
        \[
        4\vert \sqrt{x}-\sqrt{y}\vert^2\leq (x-y)(\log x-\log y).
        \]
        The second one is given by
        \[
        (\alpha+\beta)^2(y^\alpha-x^\alpha)(y^\beta-x^\beta)\geq 4\alpha\beta\left( y^{(\alpha+\beta)/2}-x^{(\alpha+\beta)/2}\right)^2\] 
        and holds for all  $\alpha,\beta >0$. We are interested in the case $\alpha=p-1$ and $\beta =1$. The reader may find a proof in \cite[Lemma 19]{chainais_jungel_schuchnigg_2016}. Altogether, it yields that for all $p\in[1,2]$, one has
        \begin{equation}\label{eq:ineg_func}
            \frac{4}{p}(x^{p/2}-y^{p/2})^2\ \leq\ (x-y)(\Phi_p'(x)-\Phi_p'(y))\,.
        \end{equation}
        As $r$ satisfies \eqref{hyp:g:bounds}, it implies that
        \[
        \frac{4}{p}(x^{p/2}-y^{p/2})^2\leq r(x,y)(\log x-\log y)(\Phi_p'(x)-\Phi_p'(y))
        \]
        for all $x,y>0$. Therefore, for all edge $\sigma\in\E_{\text{int}}$ with $\sigma =K|L$, we have 
        \be \label{I_Ihat_1}
        \frac{4}{p}\left(D_{\sigma}\left(\frac{\bu^{n+1}}{\bu^{\infty}}\right)^{p/2}\right)^2\leq r\left(\frac{u_K^{n+1}}{u_K^{\infty}},\frac{u_L^{n+1}}{u_L^{\infty}}\right)\left( D_{\sigma} \log \left(\frac{\bu^{n+1}}{\bu^{\infty}}\right)\right)\left( D_{\sigma} \Phi_p'\left(\frac{\bu^{n+1}}{\bu^{\infty}}\right)\right).
        \ee
        But thanks to the homogeneity \eqref{hyp:g:homogeneity} of $r$, we have
        \be \label{I_Ihat_2}
        {\overline u}_{\sigma}^{\infty}r\left(\frac{u_K^{n+1}}{u_K^{\infty}},\frac{u_L^{n+1}}{u_L^{\infty}}\right)=
        r\left({\overline u}_{\sigma}^{\infty}\frac{u_K^{n+1}}{u_K^{\infty}},{\overline u}_{\sigma}^{\infty}\frac{u_L^{n+1}}{u_L^{\infty}}\right)\leq r(u_K^{n+1},u_L^{n+1}),
        \ee
        because of the definition of ${\overline u}_{\sigma}^{\infty}$. We then deduce \eqref{ineg:I_Ihat} from \eqref{I_Ihat_1} and \eqref{I_Ihat_2}.

        In order to establish \eqref{ineq:EI}, we just need to prove that ${\widehat{\mathbb I}}_p^{n+1}\geq \nu_p\,{\mathbb E}_p^{n+1}$ for all $n\geq 0$. This relation is a consequence of the discrete log-Sobolev and Beckner inequalities stated in Proposition \ref{prop:func_ineq}. Indeed, let us apply \eqref{ineg:logSob_2} to $\bv=\bu^{n+1}$ and $\bv^{\infty}=\bu^{\infty}$. We get
        $
        {\mathbb E}_1^{n+1} \leq C_{LS} \left(M^\infty\,M^1\right)^{\frac{1}{2}}
        {\widehat {\mathbb I}}_1^{n+1}/(\zeta^2m^\infty).
        $
        It yields 
        \[
        \nu_1= \frac{\zeta^2}{C_{LS}}\,\frac{m^\infty}{(M^\infty\,M^1)^{\frac{1}{2}}}\,.
        \]
        Similarly, by applying \eqref{ineg:Beck_2} one gets the desired inequality with 
        \[
        \nu_p = (p-1)\,\frac{\zeta}{C_{B}}\,\frac{m^\infty}{M^\infty}\,.
        \]
        It concludes the proof.
    \end{proof}
    
    \begin{cor}
        Under the assumptions of Theorem~\ref{th:expdec_TPFA_N}, one has
        \begin{equation}\label{expdec_L2_N}
            \sum_{K\in\T}m_K\,|u_K^n-u_K^\infty|^2\ \leq\ \mathbb{E}_2^0\,M^\infty\,e^{-\tilde{\nu}_2\,t^n}
        \end{equation}
        and
        \begin{equation}\label{expdec_L1_N}
            \left(\sum_{K\in\T}m_K\,|u_K^n-u_K^\infty|\right)^2\ \leq\ 2\,\mathbb{E}_1^0\,M^1\,e^{-\tilde{\nu}_1\,t^n}\,.
        \end{equation}
    \end{cor}
    \begin{proof}
        The decay \eqref{expdec_L2_N} in $L^2$-norm (\emph{resp.} \eqref{expdec_L1_N} in $L^1$-norm )  is just a consequence of \eqref{expdec_Ep} and the Cauchy-Schwarz inequality (\emph{resp.} the Csiszár-Kullback inequality, see Lemma~\ref{lem:CK}).
    \end{proof}
    \subsection{The case of Dirichlet-Neumann boundary conditions}
    
    In this section we show the exponential decay towards the thermal equilibrium in the case of mixed Dirichlet-Neumann boundary conditions.
    
    \bt\label{th:expdec_TPFA_DN}
    Let us assume that $\E_{\text{ext}}^D\neq\emptyset$. Then, for all $p\in (1,2]$, there exists $\kappa_p$ depending only on $p$, $\Omega$, $\Gamma^D$, $\zeta$, the boundary condition $u^D$ and the potential $V$, such that, for any $k>0$, if $\Delta t\leq k$, 
    \be\label{expdec_Ep_DN}
    {\mathbb E}_p^{n}\ \leq\ (1\,+\,\kappa_p \Delta t)^{-n}\, {\mathbb E}_p^{0}\,,\quad \forall n\geq 0\,.
    \ee
    Thus for any $k>0$, if $\Delta t\leq k$, one has for all $n\geq0$ that ${\mathbb E}_p^{n}\leq e^{-{\tilde \kappa_p} t^n}{\mathbb E}_p^{0}$ with ${\tilde\kappa}_p=\log (1+\kappa_p k)/k$. 
    \et
    
    \begin{proof}
        The proof begins in the same fashion as in the case of Neumann boundary conditions. The expressions of the dissipation slightly change, as some boundary terms are taken into account. However with the same arguments one still has ${\widehat {\mathbb I}}_p^{n+1}\leq {{\mathbb I}}_p^{n+1}$ with 
        \[
        {\mathbb I}_p^{n+1}=\ds\sum_{\sigma\in\E_{\text{int}}\cup \E_{\text{ext}}^D}\tau_\sigma {\overline u}_{\sigma}^{n+1} \left(D_\sigma \log(\bu^{n+1}/\bu^{\infty})\right)\left(D_\sigma \Phi_p'(\bu^{n+1}/\bu^{\infty})\right)\,,
        \]
        and 
        \[
        {\widehat {\mathbb I}}_p^{n+1}\,=\,\frac{4}{p}\ds\sum_{\sigma\in\E_{\text{int}}\cup \E_{\text{ext}}^D}\tau_\sigma {\overline u}_{\sigma}^{\infty} \left(D_\sigma (\bu^{n+1}/\bu^{\infty})^{p/2}\right)^2\,.
        \]
        Then the proof differs as we are going to use a different functional inequality in order to establish a relation between ${\mathbb E}_p^{n+1}$ and ${\mathbb I}_p^{n+1}$ of the form \eqref{ineq:EI}. Indeed, we apply a discrete Poincar\'e inequality (see for instance \cite[Theorem 4.3]{bessemoulin_chainais_filbet_2015}). It ensures the existence of a constant $C_P$ depending only on $\Gamma^D$ and $\Omega$, such that 
        $$
        \ds\sum_{K\in\T} \mk \left(\left(\frac{u_K^n}{u_K^\infty}\right)^{\frac{p}{2}}-1\right)^2\leq \frac{(C_{P})^2}{\zeta}\ds\sum_{\sigma\in\E_{\text{int}}\cup \E_{\text{ext}}^D}\tau_\sigma \left(D_\sigma \left(\frac{\bu^{n+1}}{\bu^{\infty}}\right)^{\frac{p}{2}}\right)^2.
        $$
        Therefore, using the bounds \eqref{bound:steadystate}, we obtain:
        $$
        {\mathbb I}_p^{n+1}\geq \frac{4}{p}\frac{\zeta}{(C_P)^2}\frac{m^\infty}{(M^{\infty})^p}\ds\sum_{K\in\T} \mk \left( (u_K^{n+1})^{\frac{p}{2}}- (u_K^{\infty})^{\frac{p}{2}}\right)^2.
        $$
        But, for all $p\in (1,2]$, we have the following inequality, whose proof is left to the reader,
        $$
        (x^{p/2}-y^{p/2})^2\geq x^p-y^p-py^{p-1}(x-y)\,,
        $$
        for all $x,y>0$. It yields
        $$
        \ds\sum_{K\in\T} \mk \left( (u_K^{n+1})^{\frac{p}{2}}- (u_K^{\infty})^{\frac{p}{2}}\right)^2 \geq (p-1) (m^\infty)^{p-1}{\mathbb E}_p^{n+1}
        $$
        and finally ${\mathbb I}_p^{n+1}\geq \kappa_p {\mathbb E}_p^{n+1}$ with 
        \begin{equation}\label{eq:kappap}
            \kappa_p=\frac{4(p-1)}{p}\frac{\zeta}{(C_P)^2}\left(\frac{m^\infty}{M^{\infty}}\right)^p
        \end{equation}
        and it concludes the proof of Theorem \ref{th:expdec_TPFA_DN}.
    \end{proof}
    \begin{cor}
        Under the assumptions of Theorem~\ref{th:expdec_TPFA_DN}, one has
        \begin{equation}\label{expdec_L2_DN}
            \sum_{K\in\T}m_K\,|u_K^n-u_K^\infty|^2\ \leq\ \mathbb{E}_2^0\,M^\infty\,e^{-\tilde{\kappa}_2\,t^n}\,.
        \end{equation}
    \end{cor}
    \begin{rema}
        The restriction $p>1$ in Theorem~\ref{th:expdec_TPFA_DN} does not prevent the entropy $\mathbb{E}_1^n$ from decaying exponentially fast in time. Indeed it trivially does since $\Phi_1\leq\Phi_2$ and thus $\mathbb{E}_1^n\ \leq\ \mathbb{E}_2^n$, so that
        \[
        \mathbb{E}_1^n\  \leq\ \mathbb{E}_2^0 (1\,+\,\kappa_2 \Delta t)^{-n}\,.
        \]
        However this estimate is not as sharp as \eqref{expdec_Ep_DN}. Indeed, the difference lies in the fact that unlike \eqref{expdec_Ep_DN}, the latter estimate is not saturated at $n=0$. In the same way one could show that any sub-quadratic $\Phi$-entropy decays at least as fast as $\mathbb{E}_2^n$. The same observation suggests that the degeneracy of $\kappa_p$ (and $\nu_p$) when $p\to1$ are only technical.
    \end{rema}
    \begin{rema}
        It is unclear which functional inequality should be used in the case $p=1$ with Dirichlet-Neumann boundary conditions. This was already noticed in \cite[Remark 3.1]{bodineau_2014_lyapunov}.
    \end{rema}

    \section{Large time behavior of discrete duality finite volume (DDFV) schemes}\label{sec:DDFV}
    
    \subsection{Meshes and set of unknowns}
    
    In order to introduce the DDFV scheme from \cite{CCHK},  we need to introduce three different meshes -- the primal mesh, 
    the dual mesh and the diamond mesh -- and some associated notations.
    
    The primal mesh denoted $\overline{\M}$ is composed of the interior primal mesh $\M$ (a partition of $\Omega$ with polygonal control volumes) and the set $\dr\M$ of boundary edges seen as degenerate control volumes. 
    For all $K\in \overline{\M}$, we define $\bx_K$ the center of $K$. 
    
    To any vertex $\xke$ of the primal mesh satisfying $\xke\in \Omega$, we associate a polygonal control volume $\ke$ defined by connecting all the centers of the primal cells sharing $\xke$ as vertex. The set of such control volumes is the interior dual mesh denoted $\M^*$. To any vertex $\xke\in \partial \Omega$, we define a polygonal control volume $\ke$ by connecting the centers of gravity of the interior primal cells and the midpoints of the boundary edges sharing $\xke$ as vertex. The set of such control volumes is the boundary dual mesh, denoted $\dr\M^*$. Finally, the dual mesh is $\Me$, denoted by $\overline{\Mie}$.
    {An illustration in the case of a triangular primal mesh is provided in Figure~\ref{fig_mesh}.}
    
    \begin{figure}[htb]
        \begin{tabular}{cc}
            \begin{tikzpicture}[scale=0.6]
                \draw[line width=1pt, color=purple]  (1,1)--(1,7)--(7,7)--(7,1)--cycle;
                \draw[line width=1pt,densely dashed]  (1,1)--(2.5,3)--(3.5,1)--(5.5,3.5)--(7,1);
                \draw[line width=1pt,densely dashed] (1,7)--(3.,5)--(4.5,7)--(5.5,3.5)--(7,7);
                \draw[line width=1pt,densely dashed] (7,4)--(5.5,3.5)--(2.5,3)--(1,4.5)--(3.,5)--(2.5,3);
                \draw[line width=1pt,densely dashed] (3.,5)--(5.5,3.5);
                \draw[line width=1pt,densely dashed] (-1,5.5)--(0,5.5)--(0,4.5)--(-1,4.5)--cycle;
                \draw[line width=1pt, color=purple] (-1.,3.5)--(0,3.5);
                \node[left] at (-1.2,5){\color{black} $\M$};
                \node[left] at (-1.2,3.5){$\dr\M$};
            \end{tikzpicture}&
            \begin{tikzpicture}[scale=0.6]
                \draw[line width=1pt, color=purple]  (1,1)--(1,7)--(7,7)--(7,1)--cycle;
                \draw[line width=.5pt,densely dashed]  (1,1)--(2.5,3)--(3.5,1)--(5.5,3.5)--(7,1);
                \draw[line width=.5pt,densely dashed] (1,7)--(3.,5)--(4.5,7)--(5.5,3.5)--(7,7);
                \draw[line width=.5pt,densely dashed] (7,4)--(5.5,3.5)--(2.5,3)--(1,4.5)--(3.,5)--(2.5,3);
                \draw[line width=.5pt,densely dashed] (3.,5)--(5.5,3.5);
                \draw[line width=1.pt,color=blue](1.5,2.83)--(2.17,4.17)--(3.67,3.83)--(3.83,2.5)--(2.33,1.67)--cycle;
                \draw[line width=1pt,color=blue](2.25,1)--(2.33,1.67);
                \draw[line width=1pt,color=blue](3.83,2.5)--(5.33,1.83);
                \draw[line width=1pt,color=blue](5.25,1)--(5.33,1.83)--(6.5,2.83)--(7,2.5);
                \draw[line width=1pt,color=blue](1,2.75)--(1.5,2.83);
                \draw[line width=1pt,color=blue](3.67,3.83)--(4.33,5.17)--(5.66,5.83)--(6.5,4.83)--(6.5,2.83);
                \draw[line width=1pt,color=blue](7,5.5)--(6.5,4.83);
                \draw[line width=1pt,color=blue](5.66,5.83)--(5.75,7);
                \draw[line width=1pt,color=blue](1,5.75)--(1.67,5.5)--(2.83,6.33)--(2.75,7);
                \draw[line width=1pt,color=blue](1.67,5.5)--(2.17,4.17);
                \draw[line width=1pt,color=blue](2.83,6.33)--(4.33,5.17);
                \filldraw[color=blue,pattern=dots,pattern color=blue](1.5,2.83)--(2.17,4.17)--(3.67,3.83)--(3.83,2.5)--(2.33,1.67)--cycle;
                \filldraw[color=blue,pattern=dots,pattern color=blue](3.83,2.5)--(5.33,1.83)--(6.5,2.83)--(6.5,4.83)--(5.66,5.83)--(4.33,5.17)--(3.67,3.83)--cycle;
                \filldraw[color=blue,pattern=dots,pattern color=blue](2.17,4.17)--(3.67,3.83)--(4.33,5.17)--(2.83,6.33)--(1.67,5.5)--(2.17,4.17)--cycle;
                
                \filldraw[color=blue,pattern=dots,pattern color=blue,line width=1pt](8,5.5)--(9,5.5)--(9,4.5)--(8,4.5)--cycle;
                \draw[color=blue,line width=1pt](8,4)--(9,4)--(9,3)--(8,3);
                \draw[color=purple,line width=1pt](8,4)--(8,3);
                \node[right] at (9.2,5){\color{black} $\M^*$};
                \node[right] at (9.2,3.5){$\dr\M^*$};
                
            \end{tikzpicture}
        \end{tabular}
        \caption{{An example of primal and dual meshes}}\label{fig_mesh}
    \end{figure}
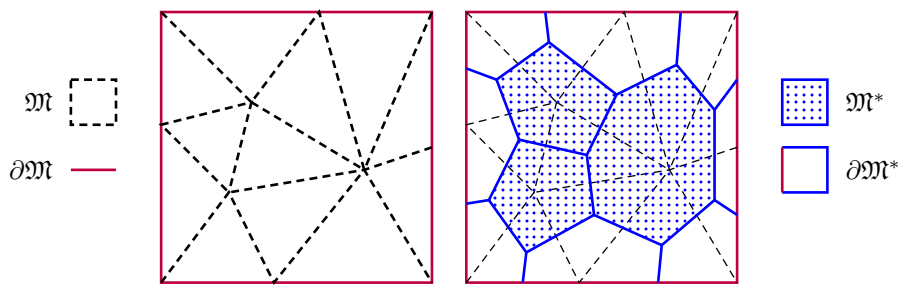

    For all neighboring primal cells $\k$ and $\l$, we assume that $\dr\k\cap\dr\l$ is a segment, corresponding to an edge of the mesh $\M$, 
    denoted by $\sigma=\k\vert\l$. Let $\Ee$ be the set of such edges. We similarly define the set $\Ee^*$ of the edges  of the dual mesh.
    For each couple $(\sigma,\sigma^*)\in\Ee\times\Ee^*$ such that $\sigma=\k\vert\l=(\xke,\xle)$ and $\sigma^*=\ke\vert\le=(\xk,\xl)$,
    we define the quadrilateral diamond cell $\Dsigsiget$ whose diagonals are $\sigma$ and $\sigma^*$. 
    If $\sigma\in \Ee\cap\dr\Omega$,
    we note that the diamond degenerates into a triangle. The set of the diamond cells defines the diamond mesh $\DD$.
    It is a partition of $\Omega$ as shown on Figure~\ref{fig_mesh_diamond}.
    We can rewrite $\DD=\DD^{\text{ext}}\cup \DD^{\text{int}}$ where $\DD^{\text{ext}}$ is the set of all the boundary diamonds and $\DD^{\text{int}}$ the set of all the interior diamonds.
    Finally, the DDFV mesh is made of  $\T=(\overline{\M},\overline{\Mie})$ and $\DD$.
    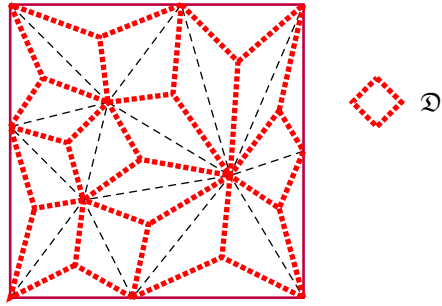
\begin{figure}[htb]
    \begin{center}
        \begin{tikzpicture}[scale=0.65]
            \draw[line width=1pt, color=purple]  (1,1)--(1,7)--(7,7)--(7,1)--cycle;
            \draw[line width=.5pt,densely dashed]  (1,1)--(2.5,3)--(3.5,1)--(5.5,3.5)--(7,1);
            \draw[line width=.5pt,densely dashed](1,7)--(3.,5)--(4.5,7)--(5.5,3.5)--(7,7);
            \draw[line width=.5pt,densely dashed] (7,4)--(5.5,3.5)--(2.5,3)--(1,4.5)--(3.,5)--(2.5,3);
            \draw[line width=.5pt,densely dashed] (3.,5)--(5.5,3.5);
            \draw[densely dotted,line width=2pt,color=red](3.5,1)--(2.33,1.67)--(2.5,3)--(3.83,2.5)--cycle;
            \draw[densely dotted,line width=2pt,color=red](2.5,3)--(1.5,2.83)--(1,4.5)--(2.17,4.17)--cycle;
            \draw[densely dotted,line width=2pt,color=red](1,7)--(1.67,5.5)--(3,5)--(2.83,6.33)--cycle;
            \draw[densely dotted,line width=2pt,color=red](4.5,7)--(5.66,5.83)--(5.5,3.5)--(4.33,5.17)--cycle;
            \draw[densely dotted,line width=2pt,color=red](6.5,4.83)--(7,4)--(6.5,2.83)--(5.5,3.5)--cycle;
            \draw[densely dotted,line width=2pt,color=red](2.17,4.17)--(3,5)--(3.67,3.83)--(5.5,3.5)--(5.33,1.83)--(7,1);
            \draw[densely dotted,line width=2pt,color=red](3.5,1)--(5.33,1.83);
            \draw[densely dotted,line width=2pt,color=red](5.5,3.5)--(3.83,2.5);
            \draw[densely dotted,line width=2pt,color=red](3.67,3.83)--(2.5,3);
            \draw[densely dotted,line width=2pt,color=red](2.33,1.67)--(1,1)--(1.5,2.83);
            \draw[densely dotted,line width=2pt,color=red](5.66,5.83)--(7,7)--(6.5,4.83);
            \draw[densely dotted,line width=2pt,color=red](4.5,7)--(2.83,6.33);
            \draw[densely dotted,line width=2pt,color=red](1,4.5)--(1.67,5.5);
            \draw[densely dotted,line width=2pt,color=red](3,5)--(4.33,5.17);
            \draw[densely dotted,line width=2pt,color=red](7,1)--(6.5,2.83);
            
            \draw[densely dotted,line width=2pt,color=red](8.5,5.5)--(9,5.)--(8.5,4.5)--(8.,5.)--cycle;
            \node[right] at (9.2,5){\color{black} $\DD$};
        \end{tikzpicture}
        \end{center}
        \caption{{An example of diamond mesh $\DD$}}\label{fig_mesh_diamond}
    \end{figure}

    For a diamond $\Dsigsiget$, whose vertices are $(\xk,\xke,\xl,\xle)$, 
    we define:
    $\bx_\Ds$ the center of the diamond cell $\Ds$, $\msig$ the length of the primal edge $\sig$,
    $\msige$ the length of the dual edge $\sige$, {$d_\Ds = \sup_{x,y\in\Ds}d(x,y)$} the diameter of $\Ds$, $\alpha_\Ds$ the angle between $(\xk,\xl)$ and $(\xke,\xle)$. 
    We will also use two direct basis $(\tkele,\nksig)$ and $(\nkesige,\tkl)$, where $\nksig$ is the unit normal to $\sigma$, outward $\k$, $\nkesige$ is the unit normal to $\sigma^*$, outward $\ke$, $\tkele$ is the unit tangent vector to $\sigma$, oriented from $\ke$ to $\le$, $\tkl$ is the unit tangent vector to $\sigma^*$, oriented from $\k$ to $\l$. All these notations are presented on Figure~\ref{fig_diamonds}.
    
    \begin{figure}[htb]
        \begin{center}
            \begin{tikzpicture}[scale=1.]
                \node[rectangle,scale=0.8,fill=black!50] (xle) at (0,0) {};
                \node[circle,draw,scale=0.5,fill=black!5] (xl) at (2,1.3) {};
                \node[rectangle,scale=0.8,fill=black!50]  (xke) at (0,4) {};
                \node[circle,draw,scale=0.5,fill=black!5] (xk) at (-2,2.3) {};
                \draw[line width=1pt,densely dashed] (xle)--(xke);
                \draw[line width=1pt,color=blue] (xk)--(xl);
                \draw[densely dotted,line width=2pt,color=red] (xk)--(xke)--(xl)--(xle)--(xk);
                
                \node[yshift=-8pt] at (xle){$\xle$};
                \node[yshift=8pt] at (xke){$\xke$};
                \node[xshift=10pt] at (xl){$\xl$};
                \node[xshift=-10pt] at (xk){$\xk$};
                
                \draw[->,line width=1pt] (-2+3.*0.4,2.3-3.*0.1)--(-2+4.8*0.4,2.3-4.8*0.1);
                \draw[->,line width=1pt] (-2+3.*0.4,2.3-3.*0.1)--(-2+3.*0.4-1.8*0.1,2.3-3.*0.1-1.8*0.4);
                \draw[->,line width=1pt] (0,2.7)--(0,2.);
                \draw[->,line width=1pt] (0,2.7)--(.7,2.7);
                \node[right] at (0,2.3){$\tkele$};
                \node[right] at (0.,2.95){$\nksig$};
                \node at (-0.5,2.2){$\tkl$};
                \node at (-0.45,1.25){$\nkesige$};
                
                \draw[line width=1pt,densely dashed]  (2.2,4)--(2.8,4);
                \node[right,xshift=8] at (2.8,4){$\sigma=\k\vert\l$, edge of the primal mesh};
                \draw[line width=1pt, color=blue] (2.2,3.5)--(2.8,3.5);
                \node[right,xshift=8] at (2.8,3.5){$\sigma^*=\ke\vert\le$, edge of the dual mesh};
                \draw[densely dotted,line width=2pt,color=red] (2.2,3)--(2.8,3);
                \node[right,xshift=8] at (2.8,3.){Diamond $\Dsigsiget$};
                \node[rectangle,scale=0.8,fill=black!50] at (2.8,2.5) {};
                \node[circle,draw,scale=0.5,fill=black!5] at (2.8,2.) {};
                \node[right,xshift=8] at (2.8,2.5){Vertices of the primal mesh};
                \node[right,xshift=8] at (2.8,2.){Centers of the primal mesh};
                \node [right]at (0,1.5) {$\bx_\Ds$};
                \node at (0,1.8) {$\bullet$};

                \node[rectangle,scale=0.8,fill=black!50] (xle2) at (11,0) {};
                \node[circle,draw,scale=0.5,fill=black!5] (xl2) at (11,2) {};
                \node[rectangle,scale=0.8,fill=black!50]  (xke2) at (11,4) {};
                \node[circle,draw,scale=0.5,fill=black!5] (xk2) at (9.5,2.3) {};
                
                \draw[line width=1pt,color=blue] (xk2)--(xl2);
                \draw[densely dotted,line width=2pt, color=red] (xle2)--(xk2)--(xke2)--(xle2);
                \draw[line width=1pt,densely dashed] (xle2)--(xke2);
                
                \node[yshift=-8] at (xle2){$\xle$};
                \node[yshift=8] at (xke2){$\xke$};
                \node[xshift=10] at (xl2){$\xl$};
                \node[xshift=-10] at (xk2){$\xk$};
                
            \end{tikzpicture}
        \end{center}
        \caption{Definition of the diamonds $\Dsigsiget$ and related notations.}\label{fig_diamonds}
    \end{figure}
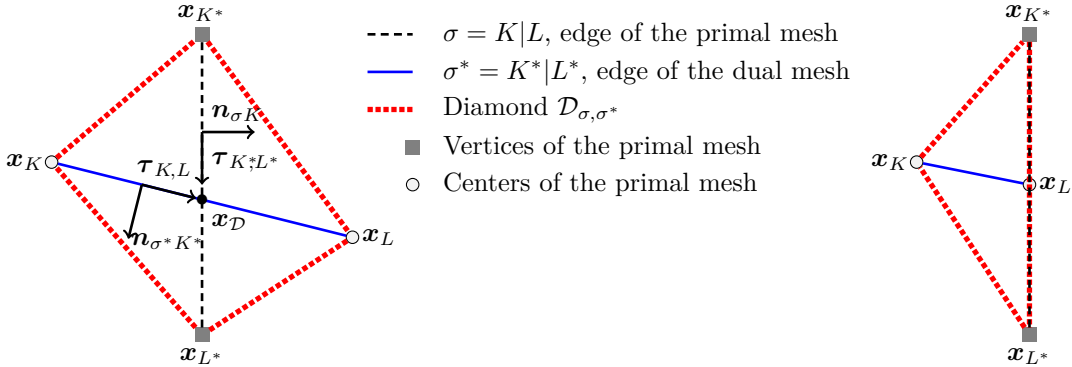

    For each primal  cell $K\in {\overline \M}$ ({\em resp.} dual cell
    $\ke\in {\overline \Mie}$), we define ${\rm m}_K$ the measure of $K$, $\Ee_K$ the set of the edges of $K$ 
    (it coincides with the edge $\sigma=K$ if $K\in\dr\M$), $\DD_K$ the set of diamonds $\Dsigsiget\in\DD$ such that ${m}(\Dsigsiget\cap K)>0$, 
    and {$d_K = \sup_{x,y\in K}d(x,y)$} the diameter of $K$ ({\em resp.} ${\rm m}_\ke$, $\Ee_\ke$, $\DD_\ke$, $d_\ke$). 
    Denoting by $\md$ the $2$-dimensional Lebesgue measure of $\Ds$, one has 
    \be\label{eq:m_Dd}
    \md = \frac12 \msig\msige\sin(\alpha_\Ds), \qquad \forall \Ds = \Ds_{\sig, \sig^\ast}\in \DD.
    \ee
    
    We assume some regularity of the mesh as presented in \cite{CCHK}. Therefore, we define two local regularity factors $\theta_\Ds, {\tilde \theta}_\Ds$ of the diamond cell $\Ds=\Ds_{\sig,\sige}\in\DD$ by 
    $$
    \theta_\Ds = \frac1{2\sin(\alpha_\Ds)}\left(\frac{\msig}{\msige} + \frac{\msige}{\msig}\right), 
    \ 
    \tilde \theta_\Ds = \max \left( \max_{\stackrel{K \in \M,}{ {\rm m}_{\Ds \cap K}>0}}  \frac{\md}{{\rm m}_{\Ds \cap K}}\, ; \,\max_{\stackrel{K^* \in \M^*,}{ {\rm m}_{\Ds \cap K^*}>0}} \frac{\md}{{\rm m}_{\Ds \cap K^*}} \right)
    $$
    and we assume that there exists ${ \Theta} \ge 1$ such that 
    \be\label{eq:Theta}
    1 \leq \theta_\Ds, \tilde \theta_\Ds \leq {\Theta}, \qquad \forall \Ds \in \DD.
    \ee
    In particular, this implies that 
    \be\label{hypsinalpha}
    \sin(\alpha_\Ds) \ge \Theta^{-1}, \qquad \forall \Ds \in \DD.
    \ee

    Let us introduce the sets of discrete unknowns. $\R^\T$ is the linear space of scalar fields constant on the primal and dual cells and $\Mie$ and $(\R^2)^\DD$ is the linear space of vector fields constant on the diamonds. We have 
    \begin{eqnarray*}
        \ut\in \R^\T& \Longleftrightarrow& \ut=\left(\left(u_K\right)_{K\in{\overline \M}},\left(u_\ke\right)_{\ke\in{\overline \Mie}}\right)\\
        {\xib}_\DD\in (\R^2)^\DD & \Longleftrightarrow&{\xib}_\DD=\left({\xib}_\Ds\right)_{\Ds\in\DD}\,.
    \end{eqnarray*}
    
    Then, we define the positive semi-definite bilinear form $\llbracket\cdot,\cdot\rrbracket_\T $ 
    on $\R^\T $ and the scalar product $(\cdot,\cdot)_{{\bm \varLambda}, \DD} $ on $(\R^2)^\DD$ by 
    $$
    \begin{aligned}
        &\ds\left\llbracket\vt,\ut\right\rrbracket_\T&&=&&\frac{1}{2}\left(\sumpri\mk\uk v_K+\sumdua \mke\uke v_\ke
        \right), \ \ \ \forall \ut,\vt\in \R^\T,\\
        &\ds\left(\xib_\DD,{\bm \phi}_\DD\right)_{{\bm \varLambda}, \DD}&&=&&\sumdiam\md\ 
        \xib_\Ds\cdot {\bm \varLambda}^\Ds {\bm \phi}_\Ds,\ \ \ \forall \xib_\DD,{\bm \phi}_\DD\in(\R^2)^\DD,
    \end{aligned}
    $$
    where 
    $${\bm \varLambda}^\Ds = \frac1{\md} \int_{\Ds} {\bm \varLambda}(\bx) \,\dd\bx, \qquad \forall \Ds \in \DD.$$
    We denote by $\|\cdot\|_{{\bm \varLambda},\DD}$ the Euclidean norm associated to the scalar product $\left(\cdot, \cdot\right)_{{\bm \varLambda},\DD}$, i.e.,
    $$
    \left\| \xib_\DD \right\|_{{\bm \varLambda},\DD}^2 = \left(\xib_\DD, \xib_\DD \right)_{{\bm \varLambda},\DD}, \qquad \forall \xib_\DD \in (\R^2)^\DD.
    $$
    Let us remark that, due to the ellipticity (A4) of ${\bm\varLambda}$, we have 
    \be\label{ineg:normesD}
    \left\| \xib_\DD \right\|_{{\bm \varLambda},\DD}^2\geq \lambda_m\left\| \xib_\DD \right\|_{{\bm I},\DD}^2 \mbox{ with } \left\Vert \xib_\DD \right\Vert_{{\bm I},\DD}^2=\ds\sumdiam \md \vert \xib_\Ds\vert^2.
    \ee
    
    \subsection{The nonlinear DDFV scheme: presentation and \emph{a priori} estimates}
    
    \subsubsection*{Discrete operators}
    
    The DDFV method is based on a discrete duality formula which links a discrete gradient operator to a discrete divergence operator, as shown in \cite{DO_05}. 
    In this paper we don't need to introduce the discrete divergence. We just define the discrete gradient. It has been introduced in \cite{CVV_99} and developed in \cite{DO_05}; it  is a mapping from $\R^\T$ to $(\R^2)^\DD$ defined 
    by  $\gradDD \ut =\ds\left(\gradD \ut\right)_{\Ds\in\DD}$ for all  $\ut\in \R^\T$, 
    where
    \[
    \gradD \ut =\frac{1}{\sin(\alpha_\Ds)}\left(\frac{\ul-\uk}{\msige}\nksig+\frac{\ule-\uke}{\msig}\nkesige\right), \quad \forall  \Ds\in\DD.
    \]
    Using \eqref{eq:m_Dd}, the discrete gradient can be equivalently written:
    \[
    \gradD \ut =\frac{1}{2\md}\left(\msig(\ul-\uk)\nksig+\msige(\ule-\uke)\nkesige\right), \quad \forall  \Ds\in\DD.
    \]
    For $\ut\in\R^\T$ and $\Ds\in\DD$, we can define $\delta^\Ds \ut$ by 
    $$
    \delta^\Ds \ut=\left(\begin{array}{c} u_K-u_L\\ \uke-\ule\end{array}\right).
    $$
    Then, we can write 
    $$
    (\gradDD \ut,\gradDD v_\T)_{{\bm \varLambda}, \DD}=\sumdiam \delta^\Ds \ut\cdot {\mathbb A}^{\Ds}\delta^\Ds v_\T,
    $$
    where the local matrices ${\mathbb A}^{\Ds}$ are defined by
    $
    {\mathbb A}^\Ds=\begin{pmatrix}
        A^\Ds_{\sig,\sig} & A^\Ds_{\sig,\sige} \\
        A^\Ds_{\sig,\sige} & A^\Ds_{\sige,\sige} 
    \end{pmatrix},
    $
    with 
    $$
    \begin{aligned}
        A^\Ds_{\sig,\sig}&= \frac{1}{4 \md}\msig^2 ({\bm \varLambda}^\Ds \n_{\k,\sig}\cdot  \n_{\k,\sig}),\\
        A^\Ds_{\sig,\sige}& =\frac{1}{4 \md}\msig {\msige}  ({\bm \varLambda}^\Ds \n_{\k,\sig}\cdot  \n_{\ke,\sige}),\\
        A^\Ds_{\sige,\sige}& =\frac{1}{4 \md}\msige^2  ({\bm \varLambda}^\Ds \n_{\ke,\sige}\cdot  \n_{\ke,\sige}).
    \end{aligned}
    $$
    
    We  also introduce a reconstruction operator on diamonds $r^\DD$. It is a mapping from  $\R^\T$ to  $\R^\DD$ defined  for all  $\ut\in\R^\T$ by  $r^\DD[\ut] =\left(r^\Ds(\ut)\right)_{\Ds\in\DD}$.
    For $\Ds\in\DD$, whose vertices are $x_K$, $x_L$, $x_\ke$, $x_\le$, we define
    \be\label{eq:rD}
    r^\Ds(\ut)=f(r(\uk,\ul),r(\uke,\ule)),
    \ee
    where $r$ satisfies the properties \eqref{hyp:g} and $f$ is either defined by $f(x,y)=\max (x,y)$ or by $f(x,y)=(x+y)/2$.
    
    \subsubsection*{Definition of the scheme}
    
    Let us first define  the discrete initial condition $\ut^0$  by taking the mean values of $u_0$ on the primal and the dual meshes. For all $K \in \M$ and $K^\ast \in \overline{\M^\ast}$, we set 
    \[
    u_K^0 \,=\, \frac1{\mk} \int_K u_0\,, \quad u_{K^\ast}^0 \,=\, \frac1{\mke} \int_{\mk} u_0\,, \quad u_{\dr\M}^0\,=\,0\,.
    \]
    The exterior potential $V$ is discretized by taking its nodal values on the primal and dual cells, namely 
    $V_K \,=\, V(\bx_K)$ and $V_{K^\ast} \,=\, V(\bx_{K^\ast})$ 
    for all $K \in \overline\M$ and $K^\ast \in \overline{\M^\ast}$.

    We can now define the nonlinear DDFV scheme, as it is introduced in \cite{CCHK} but without stabilization term. Indeed, while the stabilization term seems crucial for the proof of convergence of the scheme, the numerical experiments show that it has no significant influence on the behavior of the scheme. The scheme is the following: for all $n\geq 0$, we look for $\ut^{n+1}\in (\R_+^\ast)^\T$ solution to the variational formulation:
    \begin{subequations}\label{schemeDDFV}
        \begin{align}
            &\Bigl\llbracket\ds\frac{\ut^{n+1}-\ut^n}{\Delta t}, \psi_\T\Bigl\rrbracket_\T+T_{\DD}(\ut^{n+1}; g_\T^{n+1},\psi_\T) =0,\quad \forall \psi_\T\in\R^\T,\label{sch_formcompacte}\\
            & T_{\DD}(\ut^{n+1}; g_\T^{n+1},\psi_\T)=\sumdiam r^\Ds(\ut^{n+1}) \, \delta^\Ds g_\T^{n+1}\cdot {\mathbb A}^{\Ds}\delta^\Ds \psi_\T,\label{sch_defTd}\\
            & g_\T^{n+1}=\log (\ut^{n+1})+V_\T.\label{sch_defgt}
        \end{align}
    \end{subequations}

    \subsubsection*{Conservation of mass and steady-state} By choosing 
    successively $\psi_\T=((1)_{K\in{\overline\M}},(0)_{\ke\in\overline \Mie})$ and $\psi_\T=((0)_{K\in{\overline\M}},(1)_{\ke\in\overline\Mie})$ as test functions in \eqref{schemeDDFV}, we obtain that the mass is conserved on the primal mesh and on the dual mesh: that is for all $n\geq0$ one has
    \[
        \ds\sumpri\mk \uk^n =\sumpri\mk \uk^0=M^1 = 
        \ds\sumdua\mke \uke^n =\sumdua\mke \uke^0, 
        \qquad \forall n \geq 0.
    \]
    The steady-state $\ut^{\infty}$ associated to the scheme \eqref{schemeDDFV} is  defined for all $ K\in\M$ and $\ke \in{\overline \Mie}$ by 
    $$
    \begin{aligned}
        u_K^{\infty}&=\rho e^{-V_K}\,,\quad\mbox{with}\ \rho = M^1\left(\sumpri \mk e^{-V_K}\right)^{-1}\,,\\
        u_\ke^{\infty}&=\rho^\ast e^{-V_\ke}\,,\quad\mbox{with}\ \rho^\ast =M^1\left(\sumdua \mke e^{-V_\ke}\right)^{-1}\,.
    \end{aligned}
    $$
    Let us remark that, as for the TPFA scheme, there exists $m^\infty>0$ and $M^{\infty}>0$ (and we keep the same notations) such that for all $K\in\M$ and $\ke\in{\overline\Mie}$
    \be\label{bound:steadystate_DDFV}
    m^\infty\leq u_K^\infty\,,\ u_\ke^\infty \leq M^{\infty}\,.
    \ee
    
    \subsubsection*{Entropy-dissipation estimate and existence result}
    
    As for the TPFA scheme, we introduce the discrete relative entropy $({\mathbb E}_{1,\T}^n)_{n\geq 0}$ obtained with the Gibbs entropy $\Phi_1$. It is defined by 
    $$
    {\mathbb E}_{1,\T}^n=\llbracket \ut^{\infty}\Phi_1(\ut^n/\ut^\infty),{1}_\T\rrbracket=\llbracket \ut^{n}\log(\ut^n/\ut^\infty),{1}_\T\rrbracket,\quad \forall n\geq 0.
    $$
    where the second equality comes from the conservation of mass on each mesh. The discrete entropy dissipation is defined by 
    $$
    {\mathbb I}_{1,\T}^{n+1} =T_{\DD}(\ut^{n+1}; g_\T^{n+1},g_\T^{n+1}), \quad \forall n\geq 0.
    $$
    We notice that the definition of the steady-state implies that $\delta^\Ds g_\T^{n+1}=\delta^\Ds \log(\ds\ut^{n+1}/\ut^{\infty})$ for all $\Ds\in\DD$. Therefore 
    ${\mathbb I}_{1,\T}^{n+1}$ rewrites for all $n\geq0$ as
    $$
    {\mathbb I}_{1,\T}^{n+1}=\sumdiam r^\Ds(\ut^{n+1}) \, \delta^\Ds \log(\ut^{n+1}/\ut^{\infty})\cdot {\mathbb A}^{\Ds}\delta^\Ds \log(\ut^{n+1}/\ut^{\infty})\,.
    $$
    
    We can now state the result of existence of a discrete positive solution to the scheme \eqref{schemeDDFV}, with the discrete entropy-entropy estimate.
    
    \begin{prop}\label{theo:ex_DDFV} 
        For all $n\geq 0$, there exists a solution $\ut^{n+1}\in (\R_+^\ast)^\T$ to the nonlinear system \eqref{schemeDDFV} that satisfies the discrete entropy/entropy dissipation estimate: 
        \be\label{entropy_dissip_DDFV}
        \ds\frac{{\mathbb E}_{1,\T}^{n+1}-{\mathbb E}_{1,\T}^{n}}{\Delta t}+ {\mathbb I}_{1,\T}^{n+1}\leq 0,\mbox{ for all } n\geq 0.
        \ee
    \end{prop}
    
    The proof is a direct adaptation of the proof of \cite[Theorem 2.2]{CCHK} to the case without stabilization and with a discrete relative entropy (the entropy defined in \cite{CCHK} is a general entropy, not relative to the steady-state, which differs from ${\mathbb E}_{1,\T}^n$ from a constant). 
    Let us just mention that \eqref{entropy_dissip_DDFV} is obtained by using $\psi_\T=\log(\ut^{n+1}/\ut^\infty)$ as a test function in \eqref{sch_formcompacte}.
    
    {\begin{rema}
                 Let us precise that while an entropy inequality still holds in the case of the DDFV scheme with stabilization term as in \cite{CCHK}, the long time analysis performed in the next section does not seem to adapt to the stabilized scheme. Indeed, with the penalization term, the conservation of mass on the primal and dual meshes is not satisfied anymore which prevents the use of the discrete log-Sobolev inequality of Section~\ref{sec:logSob}.
                 
                 Nevertheless, let us recall that the stabilization term is mainly introduced to overcome a technical issue related to the identification of the limit in the convergence proof. At the practical level of numerical experiments, no noticeable difference has been observed between the behavior of the scheme with stabilization and its present version, both at the level of long time behavior and convergence. We refer to \cite{CCHK} for more details.
                \end{rema}
}

    \subsection{Analysis of the long time behavior}
    It remains to establish the exponential decay towards $0$ of the discrete relative entropy $({\mathbb E}_{1,\T}^{n})_{n\geq 0}$. As for the TPFA finite volume scheme, it is based on a relation between the discrete entropy and the discrete entropy dissipation of the form \eqref{ineq:EI}. And this inequality is a consequence of a discrete log-Sobolev inequality which is established in Proposition~\ref{prop:logSobDDFV}. 
    
    \bt\label{theo:expdec_DDFV}
    Let us assume that $\E_{\text{ext}}^D=\emptyset$. Then, there exists $\nu$ depending only on the domain $\Omega$, the regularity of the mesh $\Theta$, the initial condition $u_0$, the exterior potential $V$ and the anisotropy tensor $\varLambda$ {\em via} $\lambda_m$ and $\lambda^M$, such that
    \be\label{expdec_DDFV}
    {\mathbb E}_{1,\T}^n\ \leq\ (1\,+\,\nu \Delta t)^{-n}\,{\mathbb E}_{1,\T}^0\,,\quad\forall n\geq 0\,.
    \ee
    Thus for any $k>0$, if $\Delta t\leq k$, one has for all $n\geq0$ that ${\mathbb E}_{1,\T}^n\leq e^{-{\tilde \nu} t^n}{\mathbb E}_{1,\T}^0$ with ${\tilde\nu}=\log (1+\nu k)/k$. 
    \et 
    
    \begin{proof}
        In the following, $C$ will denote any positive constant depending only on $\Omega$, $\Theta$, $\lambda_m$ and $\lambda^M$. As for the TPFA finite volume scheme, we start with introducing a discrete counterpart of $4 u^\infty|\nabla ( u/u^\infty)^{1/2}|^2$ in the DDFV framework. We define for all $n\geq0$
        \[
        {\widehat{\mathbb I}}_{1,\T}^{n+1}=4\sumdiam {\bar u}_\Ds^\infty \delta^\Ds \sqrt{ \ut^{n+1}/\ut^\infty}\cdot {\mathbb A}^\Ds \delta^\Ds\sqrt{ \ut^{n+1}/\ut^\infty}
        \]
        with
        \[
        {\bar u}_\Ds^\infty=\min(u_K^\infty, u_L^\infty,u_\ke^\infty,u_\le^\infty)\,.
        \]
        
        In a first step, we compare ${{\mathbb I}}_{1,\T}^{n+1}$ to ${\widehat{\mathbb I}}_{1,\T}^{n+1}$. For all $\Ds\in\DD$, we introduce the diagonal matrix ${\mathbb B}^\Ds$, whose diagonal coefficients are $B_{\sig,\sig}^\Ds= |A_{\sig,\sig}^\Ds|+|A_{\sig,\sige}^\Ds|$ and $B_{\sige,\sige}^\Ds= |A_{\sige,\sige}^\Ds|+|A_{\sig,\sige}^\Ds|$. Then it is shown in \cite{CCHK} that for all $\Ds\in\DD$, there holds
        \[
        \bw\cdot {\mathbb A}^\Ds\bw \,\leq\, \bw\cdot {\mathbb B}^\Ds\bw\,\leq\, C\,\bw\cdot {\mathbb A}^\Ds\bw,\quad  \forall \bw\in\R^2\,.
        \]
        Therefore, on one hand, 
        \[
        {{\mathbb I}}_{1,\T}^{n+1}\geq C \sumdiam r^\Ds(\ut^{n+1}) \, \delta^\Ds \log(\ut^{n+1}/\ut^{\infty})\cdot {\mathbb B}^{\Ds}\delta^\Ds \log(\ut^{n+1}/\ut^{\infty})
        \]
        and on the other hand 
        \[
        {\widehat{\mathbb I}}_{1,\T}^{n+1}\leq 4\sumdiam {\bar u}_\Ds^\infty \delta^\Ds \sqrt{\ut^{n+1}/\ut^\infty}\cdot {\mathbb B}^\Ds \delta^\Ds\sqrt{\ut^{n+1}/\ut^\infty}\,.
        \]
        Besides, as ${\mathbb B}^\Ds$ is a diagonal matrix, for all $\Ds\in\DD$ we have 
        \[
        \delta^\Ds\sqrt{ \ds\frac{\ut^{n+1}}{\ut^\infty}}\cdot {\mathbb B}^\Ds \delta^\Ds\sqrt{ \ds\frac{\ut^{n+1}}{\ut^\infty}}\ =\ B_{\sig,\sig}^\Ds \left(\sqrt{\ds \frac{u_K^{n+1}}{u_K^\infty}}-
        \sqrt{\ds \frac{u_L^{n+1}}{u_L^\infty}}\right)^2+B_{\sige,\sige}^\Ds \left(\sqrt{\ds \frac{u_\ke^{n+1}}{u_\ke^\infty}}-
        \sqrt{\ds \frac{u_\le^{n+1}}{u_\le^\infty}}\right)^2.
        \]
        Adapting the inequalities \eqref{I_Ihat_1} and \eqref{I_Ihat_2} on the primal and the dual mesh, we obtain, thanks to the definition of ${\bar u}_\Ds^\infty$, 
        \begin{multline*}
            4{\bar u}_\Ds^\infty \delta^\Ds \sqrt{ \ds\frac{\ut^{n+1}}{\ut^\infty}}\cdot {\mathbb B}^\Ds \delta^\Ds\sqrt{ \ds\frac{\ut^{n+1}}{\ut^\infty}}
            \ \leq\ B_{\sig,\sig}^\Ds r(u_K^{n+1},u_L^{n+1}) \left(\log\left(\ds \frac{u_K^{n+1}}{u_K^\infty}\right)-
            \log\left(\ds \frac{u_L^{n+1}}{u_L^\infty}\right)\right)^2 \\
            + B_{\sige,\sige}^\Ds r(u_\ke^{n+1},u_\le^{n+1})\left(\log\left(\frac{u_\ke^{n+1}}{u_\ke^\infty}\right)-
            \log\left(\frac{u_\le^{n+1}}{u_\le^\infty}\right)\right)^2\,,
        \end{multline*}
        for all $\Ds\in\DD$. Moreover, the choice of the function $f$ in the reconstruction operator $r^\DD$ ensures that 
        \[
        \max(r(u_K^{n+1},u_L^{n+1}),r(u_\ke^{n+1},u_\le^{n+1}))\leq 2 r^\Ds (\ut^{n+1}), \quad \forall \Ds\in\DD
        \]
        so that we finally obtain 
        \be\label{IIhat}
        {{\mathbb I}}_{1,\T}^{n+1}\geq C{\widehat{\mathbb I}}_{1,\T}^{n+1}\quad \forall n\geq 0.
        \ee
        
        Let us now proceed with the comparison of ${\widehat{\mathbb I}}_{1,\T}^{n+1}$ and ${{\mathbb E}}_{1,\T}^{n+1}$. Thanks to the lower bound in \eqref{bound:steadystate_DDFV} and \eqref{ineg:normesD}, we have 
        \be\label{IhatFisher}
        {\widehat{\mathbb I}}_{1,\T}^{n+1}\ \geq\ 4m^\infty\left\Vert \nabla^\DD \sqrt{\ut^{n+1}/\ut^\infty}\right\Vert_{{\bm \varLambda},\DD}^2
        \ \geq\ 4m^\infty\lambda_m\left\Vert \nabla^\DD \sqrt{\ut^{n+1}/\ut^\infty}\right\Vert_{{\bm I},\DD}^2.
        \ee
        We apply the discrete log-Sobolev inequality \eqref{ineg:logSobDDFV} given in Proposition \ref{prop:logSobDDFV} with $v_\T=\ut^{n+1}$ and $v_\T^\infty=\ut^{\infty}$. It yields 
        \be\label{EFisher}
        {\mathbb E}_{1,\T}^{n+1} \leq C\, \left(M^\infty\,M^1\right)^{\frac{1}{2}} \left\Vert \nabla^\DD \sqrt{\ut^{n+1}/\ut^\infty}\right\Vert_{{\bm I},\DD}^2.
        \ee
        From \eqref{IIhat}, \eqref{IhatFisher} and \eqref{EFisher}, we obtain the expected relation between the discrete relative entropy and the discrete dissipation of the form \eqref{ineq:EI} with 
        \begin{equation}\label{eq:defnu}
            \nu=C\,m^\infty/\left(M^\infty\,M^1\right)^{\frac{1}{2}}  .
        \end{equation}
        It concludes the proof of Theorem~\ref{theo:expdec_DDFV}.
    \end{proof}

    \section{Discrete functional inequalities}\label{sec:logSob}
    
    In this section, we state and prove the various discrete functional inequalities that are needed to prove the exponential time decay of solutions to our nonlinear schemes in the case of Neumann boundary conditions. They apply to classical polygonal mesh $\M=(\T,\E,\P)$ of $\Omega$ satisfying the regularity constraint \eqref{regmesh}, but not necessarily the orthogonality property as introduced in \cite[Definition 9.1]{EGHbook}.
    
    Theorem~\ref{theo:func_ineq} is the main result of this section and constitutes a general statement of these new discrete functional inequalities. Then, in Proposition~\ref{prop:func_ineq} and Proposition~\ref{prop:logSobDDFV} we particularize these inequalities in order to use them in the long-time analysis of the present paper. Compared to previous works on discrete functional inequalities \cite{MBC_AJ_2014, bessemoulin_chainais_filbet_2015, chainais_jungel_schuchnigg_2016}, the novelty here is that the reference measure (or the steady state) is non-constant in the domain.
    
    \begin{theo}\label{theo:func_ineq}
        Let $\M=(\T,\E,\P)$ be a mesh of $\Omega$ satisfying the regularity constraint \eqref{regmesh} with parameter $\zeta>0$. Consider $(\mu_K)_{K\in\T}$ such that $\mu_K\geq0$ for all $K\in\T$ and 
        $ \sum_{K\in\T} \mk \mu_K=1$. Let us also define $\mu^\infty\ :=\ \sup_{K\in\T}\mu_K$.
        Then the following discrete functional inequalities hold.
        \begin{itemize}
            \item[\textbf{\emph{i)}}]\textbf{Discrete Poincaré-Wirtinger inequality.} There is a constant $C_{PW}>0$ depending only on $\Omega$ such that for any $(f_K)_{K\in\T}$
            \begin{equation}\label{ineg:PW}
                \sum_{K\in\T} \mk \left|f_K-\sum_{\tilde{K}\in\T}m_{\tilde{K}}f_{\tilde{K}}\mu_{\tilde{K}}\right|^2\mu_K\ \leq\  
                \frac{C_{PW}}{\zeta}\, \mu^\infty\,\sum_{\substack{\sigma\in\E_{\text{int}} \\ \sigma=K|L}}\tau_{\sigma} \left\vert f_K-f_L\right\vert^2\,.
            \end{equation}
            \item[\textbf{\emph{ii)}}]\textbf{Discrete Beckner inequality.} There is a constant $C_{B}>0$  depending only on $\Omega$ (actually $C_B=C_{PW}$) such that for all $p\in(1,2]$ and $(f_K)_{K\in\T}$ satisfying $f_K\geq0$
            \begin{equation}\label{ineg:Beck}
                \sum_{K\in\T} \mk f_K^{2}\mu_K - \left(\sum_{K\in\T} \mk f_K^{2/p}\mu_K\right)^p\ \leq\  
                \frac{C_{B}}{\zeta}\,\mu^\infty\,\sum_{\substack{\sigma\in\E_{\text{int}} \\ \sigma=K|L}}\tau_{\sigma} \left|f_K-f_L\right|^2\,.
            \end{equation}
            \item[\textbf{\emph{iii)}}]\textbf{Discrete logarithmic Sobolev inequality.} There is a constant $C_{LS}>0$ depending only on $\Omega$ such that for all $(f_K)_{K\in\T}$ satisfying $f_K>0$, one has
            \begin{equation}\label{ineg:logSob}
                \sum_{K\in\T} \mk f_K^2\log \left(\frac{f_K^2}{\sum_{\tilde{K}\in\T}m_{\tilde{K}}f_{\tilde{K}}^2\mu_{\tilde{K}}}\right)\mu_K\ \leq\  
                \frac{C_{LS}}{\zeta^2}\, \sqrt{\mu^\infty}\,\sum_{\substack{\sigma\in\E_{\text{int}} \\ \sigma=K|L}}\tau_{\sigma} \left\vert f_K-f_L\right\vert^2\,.
            \end{equation}
        \end{itemize}
    \end{theo}
    For the proof of this theorem, we need to introduce a few notations and an important technical lemma. In the following, given a sequence $(f_K)_{K\in\T}$, its piecewise constant reconstruction is denoted $f(\bx)=\sum_{K\in\T}f_K{\bm 1}_K (\bx)$. Given $\mu(\bx)\,\dd\bx$ an absolutely continuous probability measure on $\Omega$ and $g$ a bounded measurable function with respect to $\mu$, we denote by $\mu g$ the mean value of $g$ with respect to $\mu$ and ${\bar g}$ the usual mean value, namely
    \[
    \mu g=\int_\Omega g(\bx) \mu(\bx) \,\dd\bx\quad \mbox{and}\quad {\bar g} =\ds\frac{1}{m(\Omega)}\int_\Omega g(\bx) \,\dd\bx\,.
    \]
    Moreover, we denote by $\Vert\cdot\Vert_{L^q_\mu(\Omega)}$ the canonical $L^q$-norm with respect to the measure $\mu$ and $\Vert\cdot\Vert_{L^q(\Omega)}$ the canonical $L^q$-norm with respect to the Lebesgue measure.
    \begin{lem} For all $q\in[1,\infty]$ and any suitably integrable function $g$ one has
        \begin{equation}\label{eq:means_ineq}
            \Vert g-\mu g\Vert_{L^q_\mu(\Omega)}\ \leq\ 2\,\Vert g-{\bar g}\Vert_{L^q_\mu(\Omega)}\,.
        \end{equation}
    \end{lem}
    \begin{proof}
        By Fubini's theorem one has
        \[
        {\bar g} - \mu g =  \int_\Omega\int_\Omega (g(\bx)-g({\bm z}))\ds\frac{\,\dd\bx }{m(\Omega)} \mu({\bm z})d{\bm z}. 
        \]
        Therefore, applying Jensen's inequality, we obtain that $\vert {\bar g} - \mu g \vert \leq \Vert g-{\bar g}\Vert_{L^q_\mu(\Omega)}$. One then concludes via the Minkowski (triangle) inequality.
    \end{proof}
    
    \begin{proof}[Proof of Theorem~\ref{theo:func_ineq}] The inequality in \textbf{i)} is obtained from the classical discrete Poincar\'e-Wirtinger inequality (see  \cite[Theorem 3.6]{bessemoulin_chainais_filbet_2015}) which yields
        \begin{equation}\label{logSob:st4}
            \Vert f-{\bar f}\Vert_{L^2(\Omega)}^2\leq \ds\frac{C(\Omega)}{{\zeta}}\sum_{\substack{\sigma\in\E_{\text{int}} \\ \sigma=K|L}} \tau_\sigma (f_K-f_L)^2\,.
        \end{equation}
        Then, the result follows from \eqref{eq:means_ineq}.
        
        The Beckner inequality in \textbf{ii)} is obtained as in \cite{chainais_jungel_schuchnigg_2016} from the Poincaré-Wirtinger inequality \eqref{ineg:PW} and the Jensen inequality. Indeed, for $p\in[1,2]$, one has
        \[
        \|f-\mu f\|_{L^2_\mu}^2\ =\ \|f\|_{L^2_\mu}^2 - |\mu f|^2\ \geq\ \|f\|_{L^2_\mu}^2 - |\mu |f|^{2/p}|^p
        \]
        
        For the logarithmic Sobolev inequality in \textbf{iii)}, the starting point of the proof is the same as in \cite{MBC_AJ_2014}. It follows \cite[Lemma 2.1]{desvillettes_fellner_2007} adapted to a probability measure $\mu$ different from the  Lebesgue measure. Proceeding as in the previous references we get for all $q>2$
        \begin{equation}\label{logSob:startpt}
            \ds\int_\Omega f^2 \log \frac{f^2}{\Vert f\Vert^2_{L^2_\mu(\Omega)}} \mu(\bx) \,\dd\bx\ \leq\  \frac{q}{q-2}\, \Vert f-\mu f\Vert_{L^q_\mu(\Omega)}^2\,+\,\frac{q-4}{q-2}\,\Vert f-\mu f\Vert_{L^2_\mu(\Omega)}^2.
        \end{equation} 
        Thus by choosing $q=4$ and using inequality \eqref{eq:means_ineq} we get
        \be\label{logSob:st1}
        \ds\int_\Omega f^2 \log \frac{f^2}{\Vert f\Vert^2_{L^2_\mu(\Omega)}} \mu(\bx) \,\dd\bx\ \leq\ 8\,\sqrt{\mu^\infty} \Vert f-{\bar f}\Vert_{L^4(\Omega)}^2\,.
        \ee
        We may now apply a discrete Poincar\'e-Sobolev inequality to the piecewise constant function $f-{\bar f}$, see \cite[Theorem 3.2]{bessemoulin_chainais_filbet_2015}. It yields 
        \begin{equation}\label{logSob:st3}
            \Vert f-{\bar f}\Vert_{L^4(\Omega)}^2\leq \ds\frac{\tilde{C}(\Omega)}{{\zeta}}\left(\Vert f-{\bar f}\Vert_{L^2(\Omega)}^2+
            \ds\sum_{\substack{\sigma\in\E_{\text{int}} \\ \sigma=K|L}} \tau_\sigma \left(f_K-f_L)^2\right)\right),
        \end{equation}
        where $\tilde{C}(\Omega)$ depends only on $\Omega$. Finally, as $\zeta\leq 1$, we deduce from \eqref{logSob:st3} and \eqref{logSob:st4} the existence of some $\hat{C}(\Omega)$ (depending on $C(\Omega)$ and $\tilde{C}(\Omega)$) such that  
        \[
        \Vert f-{\bar f}\Vert_{L^4(\Omega)}^2\leq \ds\frac{\hat{C}(\Omega)}{{\zeta}^2}\sum_{\substack{\sigma\in\E_{\text{int}} \\ \sigma=K|L}} \tau_\sigma (f_K-f_L)^2\,.
        \]
        Together with \eqref{logSob:st1}, it proves the result.
    \end{proof}
    
    Let us now particularize Theorem~\ref{theo:func_ineq} in order to fit with the objects of the previous sections. 
    
    \begin{prop}\label{prop:func_ineq}
        Let $\M=(\T,\E,\P)$ be a mesh of $\Omega$ satisfying the regularity constraint \eqref{regmesh}. Consider $(v_K)_{K\in\T}$ and $(v_K^\infty)_{K\in\T}$ satisfying $v_K>0$, $v_K^\infty>0$ for all $K\in\T$ and 
        \[\ds\sum_{K\in\T} \mk v_K=\ds\sum_{K\in\T} \mk v_K^{\infty}=:M^1\,,\]
        and write $M^\infty\ =\ \sup_{K\in\T}v_K^\infty$. There holds
        \be\label{ineg:logSob_2}
        \ds\sum_{K\in\T} \mk \Phi_1\left(\frac{v_K}{v_K^\infty}\right) v_K^\infty\leq 
        \frac{C_{LS}}{\zeta^2}  \left(M^1\,M^\infty\right)^\frac{1}{2}\sum_{\substack{\sigma\in\E_{\text{int}} \\ \sigma=K|L}}\tau_{\sigma} \left\vert \sqrt{\frac{v_K}{v_K^\infty}}-\sqrt{\frac{v_L}{v_L^\infty}}\right\vert^2\,,
        \ee
        and for all $p\in(1,2]$
        \be\label{ineg:Beck_2}
        \ds\sum_{K\in\T} \mk \Phi_p\left(\frac{v_K}{v_K^\infty}\right) v_K^\infty\ \leq\ 
        \frac{C_{B}}{(p-1)\zeta}\,  M^\infty\sum_{\substack{\sigma\in\E_{\text{int}} \\ \sigma=K|L}}\tau_{\sigma} \left\vert \left(\frac{v_K}{v_K^\infty}\right)^{\frac p2}-\left(\frac{v_L}{v_L^\infty}\right)^{\frac p2}\right\vert^2\,,
        \ee
        where we recall that $\Phi_p(s) = (s^p-ps)/(p-1) + 1$ and $\Phi_1(s) = s\log(s)-s+1$.
    \end{prop}
    \begin{proof} The first inequality is a consequence of \eqref{ineg:logSob} with the choice  $f_K = (v_K/v_K^\infty)^{1/2}$ and $\mu_K = v_K^\infty/M^1$. The second inequality is obtained by taking $f_K = (v_K/v_K^\infty)^{p/2}$ and $\mu_K = v_K^\infty/M^1$ in \eqref{ineg:Beck}.
    \end{proof}
    \begin{rema}
        Observe that when $p\to1$ the left-hand side of \eqref{ineg:Beck_2} degenerates to the left-hand side of \eqref{ineg:logSob_2}. However, the right-hand side of \eqref{ineg:Beck_2} tends to $+\infty$ in the same limit. It suggests that the constant in the discrete Beckner inequality \eqref{ineg:Beck} is far from optimal with respect to its dependence in $p$.
    \end{rema}

    From Proposition \ref{prop:func_ineq}, we may now deduce a discrete log-Sobolev inequality which applies to some DDFV reconstruction on primal and dual meshes $\T=({\overline \M}, {\overline \Mie})$ associated to a diamond mesh $\DD$.
    \begin{prop}\label{prop:logSobDDFV}
        Let $\T=({\overline \M}, {\overline \Mie})$ be a mesh of $\Omega$, associated to a diamond mesh $\DD$, satisfying the regularity constraint \eqref{eq:Theta}. Consider $\vt\in\R^\T$ and $\vt^{\infty}\in\R^\T$ satisfying $v_K,\ v_K^\infty>0$ for all $K\in\M$ and $v_\ke,\ v_\ke^\infty>0$ for all $\ke\in{\overline\Mie}$ and 
        $$
        \begin{gathered}
            \ds\sumpri\mk v_K\ =\ \sumpri\mk v_K^\infty\ =\ \sumdua\mke v_\ke\ =\ \sumdua\mke v_\ke^\infty\ =:\ M^1\,.
        \end{gathered}
        $$
        We write $M^\infty:=\sup_{K,\ke}\max(v_K^\infty,v_{\ke}^\infty)$. Then there exists a constant $C$ depending only on $\Omega$ and the regularity of the mesh $\Theta$ such that
        \be\label{ineg:logSobDDFV}
        \left\llbracket v_\T\log \frac{v_\T}{v_\T^{\infty}},{\bm 1}_\T\right\rrbracket\ \leq\  
        C \left(M^1\,M^\infty\right)^\frac{1}{2}\sumdiam
        \md \left\vert \nabla^\Ds \sqrt{\frac{v_\T}{v_\T^{\infty}}}\right\vert^2.
        \ee
    \end{prop}
    
    \begin{proof} In the following, $C$ will denote any positive constant depending only on $\Omega$ and $\Theta$. Let us first notice that the regularity constraint \eqref{eq:Theta} implies that the primal mesh ${\M}$ and the dual mesh ${ \overline \Mie}$ both satisfy the regularity constraint \eqref{regmesh} with $\zeta=1/\Theta^2 $. Indeed, for all $\Ds=\Ds_{\sig,\sige}$, we note that ${\rm d}_\sig=\msige$ and ${\rm d}_{\sige}= \msig$. Moreover, for all $K\in\M$ such that $K\cap \Ds\neq \emptyset$, we have 
        $$
        \md=\ds\frac{1}{2}\msig\msige \sin (\alpha_\Ds)\ \mbox{ and }\ {\rm m}_{\Ds\cap K} = \ds\frac{1}{2} \msig d(\bx_K,\sig)\,,
        $$
        {where $d(\bx_K,\sig)$ denotes the length of the altitude in the triangle with base $\sigma$ and the vertex $x_K$}. As $\md\leq \Theta {\rm m}_{\Ds\cap K}$ by \eqref{eq:Theta}, using \eqref{hypsinalpha}, we obtain that $$
        d(\bx_K,\sig)\geq \ds\frac{1}{\Theta^2} \msige,
        $$
        which corresponds to \eqref{regmesh} for $\M$. Similarly, for all $\Ds=\Ds_{\sig,\sige}$ and for all $\ke\in{\overline\Mie}$ such that $\ke\cap \Ds\neq \emptyset$, we can prove that 
        $$
        d(\bx_\ke,\sige)\geq \ds\frac{1}{\Theta^2} \msig,
        $$
        using the fact that $\md\leq \Theta {\rm m}_{\Ds\cap \ke}$. {Here again, $d(\bx_\ke,\sige)$ denotes the length of the altitude in the triangle with base $\sige$ and the vertex $x_\ke$}. 
        
        Therefore, we can apply Proposition \ref{prop:func_ineq} twice to get
        \begin{gather}
            \sumpri \mk v_K\log\left(\frac{v_K}{v_K^\infty}\right)\ \leq\ C\,(M^1\,M^\infty)^{\frac{1}{2}} 
            \sum_{\Ds_{\sig,\sige}\in\DD} \ds\frac{\msig}{\msige} \left\vert \sqrt{\frac{v_K}{v_K^\infty}}-\sqrt{\frac{v_L}{v_L^\infty}}\right\vert^2,\label{inegprim}\\
            \sumdua \mke v_\ke\log\left(\frac{v_\ke}{v_\ke^\infty}\right)\ \leq\ C\,(M^1\,M^\infty)^{\frac{1}{2}} 
            \sum_{\Ds_{\sig,\sige}\in\DD} \ds\frac{\msige}{\msig} \left\vert \sqrt{\frac{v_\ke}{v_\ke^\infty}}-\sqrt{\frac{v_\le}{v_\le^\infty}}\right\vert^2\,.\label{inegdua}
        \end{gather}
        But, the definition of the discrete gradient $\nabla^\DD$ implies that for all $f_\T\in \R^\T$ and for all $\Ds_{\sig,\sige}\in\DD$, 
        $$
        \frac{f_K-f_L}{\msige}= \nabla^\Ds f_\T \cdot \tau_{K,L}\ \mbox{ and }\ \frac{f_\ke-f_\le}{\msig}= \nabla^\Ds f_\T \cdot \tau_{\ke,\le}.
        $$
        Therefore,
        \be\label{diff2grad}
        \left\vert\frac{f_K-f_L}{\msige}\right\vert^2\leq \left\vert \nabla^\Ds f_\T\right\vert^2 
        \ \mbox{ and }\  
        \left\vert\frac{f_\ke-f_\le}{\msig}\right\vert^2\leq \left\vert \nabla^\Ds f_\T\right\vert^2,
        \ee
        which can be written for $f_\T=\sqrt{v_\T/v_\T^\infty}$. As $\msig\msige\leq 2\Theta \md$ thanks to \eqref{eq:Theta}, we obtain \eqref{ineg:logSobDDFV} 
        by summing \eqref{inegprim} and \eqref{inegdua}.
    \end{proof}
    
    Finally, we state the Csiszár-Kullback inequality. Its proof is rather classical, but as it is hard to find a reference in the literature, we briefly recall it here. This version is greatly inspired by a course of Stéphane Mischler.
    \begin{lem}\label{lem:CK}
        Let $\mu$ be a probability measure and $g$ a positive measurable function such that $\int g\dd\mu = 1$. Then
        \begin{equation}\label{eq:CK}
            \left(\int |g-1|\,\dd\mu\right)^2\ \leq\ 2\int g\log(g)\,\dd\mu
        \end{equation}
    \end{lem}
    \begin{proof}
        Let $\varphi(g)\ =\ (2g+4)(g\log(g)-g+1)-3(g-1)^2$. A direct computation of the derivatives shows that $\varphi''(g)\geq0$ for all $g\geq0$ and $\varphi'(1) = \varphi(1) = 0$. Thus $\varphi(g)\geq0$ for all $g\geq0$. By the Cauchy-Schwarz inequality it yields 
        \[
        \left(\int |g-1|\,\dd\mu\right)^2\ \leq\  \left(\int \left(\frac{2g}{3}+\frac{4}{3}\right)\,\dd\mu\right)\left(\int (g\log(g)-g+1)\,\dd\mu\right)\,,
        \]
        which is exactly \eqref{eq:CK} since $\mu$ and $g\mu$ are probability measures.
    \end{proof}

    \section{Numerical experiments}\label{sec:num_exp}
   
   \subsection{Numerical resolution of the nonlinear systems}
   
   Even though the continuous problem \eqref{FPmodel} is linear, all the schemes studied in this paper are nonlinear, in the sense that $\bu^n$ in the TPFA context or $u_\T^n$ in the DDFV context solve nonlinear systems of the form 
   \[
   \boldsymbol{\Phi}_\T(\bu^n) = \bu^{n-1} \quad \text{or}\quad 
   \boldsymbol{\Phi}_\T(u_\T^n) = u_\T^{n-1}
   \]
   for some $\boldsymbol\Phi_\T:(0,+\infty)^\T \to \R^\T$ which is singular near the boundary of its domain. Our resolution strategy relies on Newton-Raphson method, but since there is no guaranty that 
   the Newton iterations remain in $(0,+\infty)^\T$, we project them on $[10^{-12},+\infty)^\T$. As a stopping criterion, we choose 
   \[\|\Delta t^{-1}\mathbb M(\boldsymbol{\Phi}_\T(\bu^n) - \bu^{n-1})\|_{\ell^1}\leq 10^{-10},\]
   where $\mathbb M$ is the diagonal mass matrix, the diagonal entry of which being given by ${\rm m}_K$ in the TPFA context, and 
   $\left(\left({\rm m}_K\right)_{K\in\overline{\mathfrak M}}, \left({\rm m}_{K^*}\right)_{K^*\in\overline{\mathfrak M}^*}\right)$ in the DDFV context.
    \subsection{Long time behavior of TPFA schemes}
    
    In the following the domain is set to $\Omega = [0,1]\times[0,1]$ and the simulations are performed on a family of meshes generated from 
    the family of triangular meshes from~\cite{bench_FVCA5}
    with size $h_0 = 0.25$. Each refinement divides the size of the mesh by $2$. Our first test case is taken from \cite{chainais_herda, filbet_herda_2017} and consists in the resolution of \eqref{FPmodel} 
    in the case where ${\bm \varLambda}={\bm I}$, with a potential given by $V(x_1,x_2) = x_1$ and endowed with the boundary conditions $u^D(0,x_2) = 1$, $u^D(1,x_2) = \exp(1)$ on the left and right edges, with $\Gamma^D = (\{0\}\times[0,1])\cup(\{1\}\times[0,1]$), and no-flux boundary conditions on the top and bottom edges, with  $\Gamma^N = ([0,1]\times\{0\})\cup([0,1]\times\{1\}$). The exact solution of this equation is given by  
    \be\label{explicit}
    u_{ex}(x_1,x_2,t)\ =\ \exp(x_1) + \exp\left(\frac{x_1}{2}-\left(\pi^2+\frac 14\right)t\right)\sin(\pi x_1)\,.
    \ee
    We solve the equation with our non-linear scheme \eqref{scheme} for various functions $r$ and meshes. In Table~\ref{tb:cvTPFA}, we show the convergence result. The errors between the numerical and the exact solutions are computed in discrete $L^2$ norm at final time $T=0.1$. The time step is appropriately refined in order to observe the space discretization error only. One notices that all schemes are of order $2$ in space except for the less regular $r(x,y)=\max(x,y)$ for which the order is only $1$.  All the three second order schemes produce very similar results, but it is worth mentioning that the scheme corresponding to $r(x,y) = (x+y)/2$ is 
    the easiest to implement regarding the assembling of the Jacobian matrix. 
    On Figure~\ref{fig:decay}, we illustrate the exponential decay of the numerical solutions towards the steady state. The simulation is performed on the coarsest mesh with $\Delta t = 10^{-4}$. We find a good agreement between the theoretical and experimental rate of decay and observe that the choice $r(x,y) = \max(x,y)$ tends to overdissipate. This is in agreement with inequality \eqref{eq:ineg_r}. Moreover, we observe on the interval $t\in [0, 0.5]$ a maximal number of {Newton} iterations of $2$ for all schemes, while the mean number of {Newton} iterations is $1.69$ for $r(x,y) = (x+y)/2$, $1.58$ for $r(x,y) ={(y-x)}/{\log(y/x)}$, $1.62$ for $r(x,y)= {(\sqrt{x} + \sqrt{y})^2}/{4}$ and $1.93$ for $r(x,y)=\max(x,y)$. After $t=0.5$, the number of {Newton} iterations is always equal to $1$.
    \begin{table}[!h]
        \[
        \begin{array}{|c|cc|cc|cc|cc|}
            \hline	&&&&&&&&\\
            &\scriptstyle r(x,y)\ \,=& \frac{x+y}{2} &\scriptstyle r(x,y)\ \,=& \frac{y-x}{\log(y/x)}&\scriptstyle r(x,y)\ \,=& \frac{(\sqrt{x} + \sqrt{y})^2}{4}&\scriptstyle r(x,y)\ \,=&\scriptstyle \mathrm{max}(x, y)\\[1em]\hline
            \text{Size}(\T)&\text{Error}&\text{Order}&\text{Error}&\text{Order}&\text{Error}&\text{Order}&\text{Error}&\text{Order}\\
            \hline
            h_0   &1.94\cdot10^{-2}&    &1.98\cdot10^{-2}&    &1.97\cdot10^{-2}&    &6.64\cdot10^{-3}&    \\
            h_0/2 &4.94\cdot10^{-3}&1.97&5.08\cdot10^{-3}&1.96&5.05\cdot10^{-3}&1.96&2.86\cdot10^{-3}&1.21\\
            h_0/4 &1.24\cdot10^{-3}&2.00&1.28\cdot10^{-3}&1.99&1.27\cdot10^{-3}&1.99&1.35\cdot10^{-3}&1.08\\
            h_0/8 &3.10\cdot10^{-4}&2.00&3.20\cdot10^{-4}&2.00&3.17\cdot10^{-4}&2.00&6.77\cdot10^{-4}&1.00\\
            h_0/16&7.74\cdot10^{-5}&2.00&8.00\cdot10^{-5}&2.00&7.93\cdot10^{-5}&2.00&3.41\cdot10^{-4}&0.99\\
            h_0/32&1.94\cdot10^{-5}&2.00&2.00\cdot10^{-5}&2.00&1.98\cdot10^{-5}&2.00&1.71\cdot10^{-4}&1.00\\
            \hline
        \end{array}
        \]
        \caption{\textbf{TPFA.} Error in $L^2$ between the numerical and the exact solution at final time and experimental order of convergence.}\label{tb:cvTPFA}
    \end{table}

    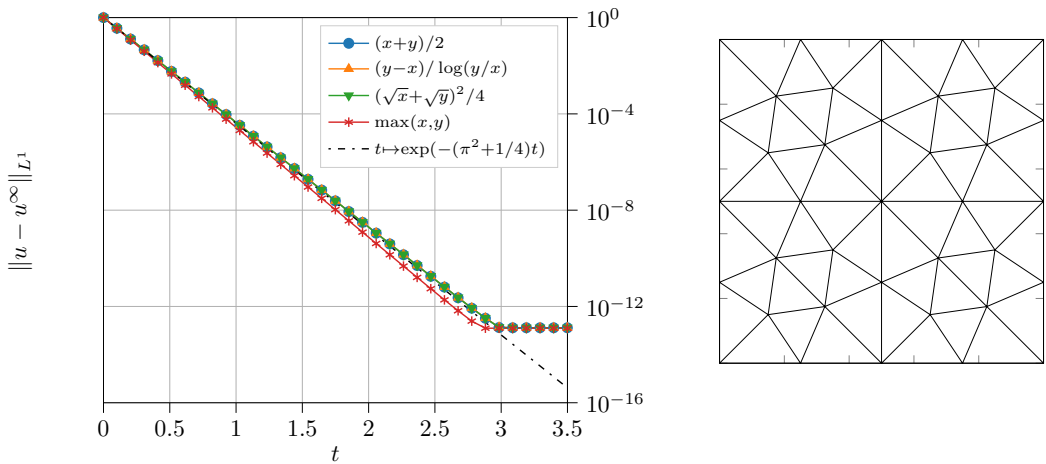
\begin{figure}
    \begin{minipage}{0.58\textwidth}
\begin{tikzpicture}[scale=0.9]

\definecolor{color0}{rgb}{0.12156862745098,0.466666666666667,0.705882352941177}
\definecolor{color1}{rgb}{1,0.498039215686275,0.0549019607843137}
\definecolor{color2}{rgb}{0.172549019607843,0.627450980392157,0.172549019607843}
\definecolor{color3}{rgb}{0.83921568627451,0.152941176470588,0.156862745098039}

\begin{axis}[
legend cell align={left},
legend style={draw=white!80.0!black},
log basis y={10},
tick align=outside,
ytick pos=right,
xtick pos=left,
xlabel={$t$},
ylabel={$\|u-u^\infty\|_{L^1}$},
x grid style={white!69.01960784313725!black},
xmajorgrids,
xmin=0, xmax=3.5,
xtick style={color=black},
y grid style={white!69.01960784313725!black},
ymajorgrids,
ymin=1e-16, ymax=1,
ymode=log,
ytick style={color=black}
]
\addplot [semithick, color0, mark=*, mark size=2, mark options={solid}]
table {%
0 1
0.101000000000002 0.361889303218341
0.203999999999994 0.128901014961485
0.306999999999983 0.0459151839243606
0.409999999999971 0.0163549669132619
0.51299999999996 0.00582557076528877
0.615999999999948 0.00207503477590773
0.718999999999937 0.000739114215568106
0.821999999999926 0.000263267633053555
0.924999999999914 9.37741828713521e-05
1.0279999999999 3.34017387239161e-05
1.13099999999989 1.18974763087198e-05
1.23399999999988 4.2378015793688e-06
1.33699999999987 1.50947660672789e-06
1.43999999999986 5.37665480718164e-07
1.54299999999985 1.91512851175411e-07
1.64599999999984 6.82155976154627e-08
1.74899999999982 2.42979400145502e-08
1.85099999999981 8.74193933068558e-09
1.9539999999998 3.11382009791576e-09
2.05699999999992 1.10912174777816e-09
2.16000000000013 3.95061465994849e-10
2.26300000000035 1.40718304630873e-10
2.36600000000057 5.0124512208905e-11
2.46900000000079 1.78592669433293e-11
2.572000000001 6.37286467227268e-12
2.67500000000122 2.29464565213279e-12
2.77800000000144 8.44510555117648e-13
2.88100000000166 3.23294740272914e-13
2.98400000000187 1.33268348290729e-13
3.08700000000209 1.30393582948181e-13
3.19000000000231 1.30393582948181e-13
3.29300000000252 1.30393582948181e-13
3.39600000000274 1.30393582948181e-13
3.49900000000296 1.30393582948181e-13
};
\addlegendentry{$\scriptstyle (x+y)/2$}
\addplot [semithick, color1, mark=triangle*, mark size=2, mark options={solid}]
table {%
0 1
0.101000000000002 0.362421998729404
0.203999999999994 0.12921518445797
0.306999999999983 0.0460679823765172
0.409999999999971 0.0164240213550685
0.51299999999996 0.00585541757274827
0.615999999999948 0.00208754355562219
0.718999999999937 0.000744239935385511
0.821999999999926 0.000265332412587428
0.924999999999914 9.45948754980726e-05
1.0279999999999 3.37244520947301e-05
1.13099999999989 1.20232586865423e-05
1.23399999999988 4.2864669397531e-06
1.33699999999987 1.52818792968367e-06
1.43999999999986 5.44821266636268e-07
1.54299999999985 1.94236720755075e-07
1.64599999999984 6.92482212675971e-08
1.74899999999982 2.46879999812774e-08
1.85099999999981 8.89020890252831e-09
1.9539999999998 3.16948876975642e-09
2.05699999999992 1.12996826802885e-09
2.16000000000013 4.02850063436572e-10
2.26300000000035 1.43621529827325e-10
2.36600000000057 5.1204604607533e-11
2.46900000000079 1.82600809014715e-11
2.572000000001 6.5219769537561e-12
2.67500000000122 2.34878120129763e-12
2.77800000000144 8.64878193622298e-13
2.88100000000166 3.31124145239017e-13
2.98400000000187 1.3696435373557e-13
3.08700000000209 1.30608965107072e-13
3.19000000000231 1.30608965107072e-13
3.29300000000252 1.30608965107072e-13
3.39600000000274 1.30608965107072e-13
3.49900000000296 1.30608965107072e-13
};
\addlegendentry{$\scriptstyle (y-x)/\log(y/x)$}
\addplot [semithick, color2, mark=triangle*, mark size=2, mark options={solid,rotate=180}]
table {%
0 1
0.101000000000002 0.362288351526679
0.203999999999994 0.129136386129574
0.306999999999983 0.0460296648695491
0.409999999999971 0.0164067041885758
0.51299999999996 0.00584793175572452
0.615999999999948 0.00208440564746982
0.718999999999937 0.000742953817439347
0.821999999999926 0.000264814196367968
0.924999999999914 9.43888426160154e-05
1.0279999999999 3.36434125726297e-05
1.13099999999989 1.19916630235655e-05
1.23399999999988 4.27423885881912e-06
1.33699999999987 1.52348492082539e-06
1.43999999999986 5.43022132885085e-07
1.54299999999985 1.93551660627025e-07
1.64599999999984 6.89884317012025e-08
1.74899999999982 2.45898365538123e-08
1.85099999999981 8.85288300235873e-09
1.9539999999998 3.15546976296625e-09
2.05699999999992 1.12471732902224e-09
2.16000000000013 4.00887709107187e-10
2.26300000000035 1.42889355262552e-10
2.36600000000057 5.09322383131813e-11
2.46900000000079 1.81590671789329e-11
2.572000000001 6.48428235604915e-12
2.67500000000122 2.33523830433442e-12
2.77800000000144 8.59858616362835e-13
2.88100000000166 3.29274272584673e-13
2.98400000000187 1.36184422129419e-13
3.08700000000209 1.30608965107072e-13
3.19000000000231 1.30608965107072e-13
3.29300000000252 1.30608965107072e-13
3.39600000000274 1.30608965107072e-13
3.49900000000296 1.30608965107072e-13
};
\addlegendentry{$\scriptstyle (\sqrt{x} + \sqrt{y})^2/4$}
\addplot [semithick, color3, mark=asterisk, mark size=2, mark options={solid}]
table {%
0 1
0.101000000000002 0.346085355861637
0.203999999999994 0.117552376381051
0.306999999999983 0.0398653326254933
0.409999999999971 0.0135134093632018
0.51299999999996 0.00458005387658152
0.615999999999948 0.00155222529804569
0.718999999999937 0.000526055620962528
0.821999999999926 0.000178281432440442
0.924999999999914 6.0419862332861e-05
1.0279999999999 2.04763745869445e-05
1.13099999999989 6.93947003083288e-06
1.23399999999988 2.35179526110708e-06
1.33699999999987 7.97026399044981e-07
1.43999999999986 2.70113255836555e-07
1.54299999999985 9.15417238820296e-08
1.64599999999984 3.10236059808331e-08
1.74899999999982 1.05139390530478e-08
1.85099999999981 3.60081660648361e-09
1.9539999999998 1.22032106124453e-09
2.05699999999992 4.1356778480189e-10
2.16000000000013 1.40158340086698e-10
2.26300000000035 4.75010300182516e-11
2.36600000000057 1.61035419614861e-11
2.46900000000079 5.4700306004555e-12
2.572000000001 1.87694722867912e-12
2.67500000000122 6.62244890601809e-13
2.77800000000144 2.45614538264415e-13
2.88100000000166 1.24883403549365e-13
2.98400000000187 1.24883403549365e-13
3.08700000000209 1.24883403549365e-13
3.19000000000231 1.24883403549365e-13
3.29300000000252 1.24883403549365e-13
3.39600000000274 1.24883403549365e-13
3.49900000000296 1.24883403549365e-13
};
\addlegendentry{$\scriptstyle \max(x, y)$}
\addplot [semithick, black, dash pattern=on 1pt off 3pt on 3pt off 3pt]
table {%
0 1
0.101000000000002 0.3598456661792
0.203999999999994 0.126894493476664
0.306999999999983 0.0447475514869221
0.409999999999971 0.0157795922361516
0.51299999999996 0.00556445040824157
0.615999999999948 0.00196222487136533
0.718999999999937 0.000691950896013362
0.821999999999926 0.000244006713746598
0.924999999999914 8.60455224445068e-05
1.0279999999999 3.03427386036476e-05
1.13099999999989 1.06999383560378e-05
1.23399999999988 3.77318218762383e-06
1.33699999999987 1.33055942448192e-06
1.43999999999986 4.69202994725408e-07
1.54299999999985 1.65457811360069e-07
1.64599999999984 5.83463610586831e-08
1.74899999999982 2.05750204285115e-08
1.85099999999981 7.32928590684838e-09
1.9539999999998 2.58456919203815e-09
2.05699999999992 9.11411833748016e-10
2.16000000000013 3.21396514844321e-10
2.26300000000035 1.13335943126052e-10
2.36600000000057 3.99663201403835e-11
2.46900000000079 1.40935585085052e-11
2.572000000001 4.96989441947538e-12
2.67500000000122 1.75256309652574e-12
2.77800000000144 6.18016631352928e-13
2.88100000000166 2.17934839199789e-13
2.98400000000187 7.68516439971252e-14
3.08700000000209 2.71006471785195e-14
3.19000000000231 9.55666059039769e-15
3.29300000000252 3.37002142563038e-15
3.39600000000274 1.18839047403432e-15
3.49900000000296 4.1906912164849e-16
};
\addlegendentry{$\scriptstyle t \mapsto \exp(-(\pi^2+1/4)t)$}
\end{axis}

\end{tikzpicture}
    \end{minipage}
    \begin{minipage}{0.38\textwidth}
      \vspace*{-.8cm}
%
%
\begin{tikzpicture}[scale=0.7]

\begin{axis}[%
width=\linewidth,
height=\linewidth,
at={(0.758in,0.499in)},
scale only axis,
unbounded coords=jump,
xmin=0,
xmax=1,
xlabel style={font=\color{white!15!black}},
ymin=0,
ymax=1,
ylabel style={font=\color{white!15!black}},
axis background/.style={fill=white},
xticklabels={,,},
yticklabels={,,}
]
\addplot [color=black, forget plot]
  table[row sep=crcr]{%
0	0.5\\
0.25	0.5\\
nan	nan\\
0	0.5\\
0	0.75\\
nan	nan\\
0	0.5\\
0.15	0.65\\
nan	nan\\
0	0.5\\
0	0.25\\
nan	nan\\
0	0.5\\
0.175	0.325\\
nan	nan\\
0.25	0.5\\
0.5	0.5\\
nan	nan\\
0.25	0.5\\
0.15	0.65\\
nan	nan\\
0.25	0.5\\
0.325	0.675\\
nan	nan\\
0.25	0.5\\
0.35	0.35\\
nan	nan\\
0.25	0.5\\
0.175	0.325\\
nan	nan\\
0.5	0.5\\
0.5	0.75\\
nan	nan\\
0.5	0.5\\
0.325	0.675\\
nan	nan\\
0.5	0.5\\
0.75	0.5\\
nan	nan\\
0.5	0.5\\
0.65	0.65\\
nan	nan\\
0.5	0.5\\
0.5	0.25\\
nan	nan\\
0.5	0.5\\
0.35	0.35\\
nan	nan\\
0.5	0.5\\
0.675	0.325\\
nan	nan\\
0.5	0.75\\
0.5	1\\
nan	nan\\
0.5	0.75\\
0.325	0.675\\
nan	nan\\
0.5	0.75\\
0.35	0.85\\
nan	nan\\
0.5	0.75\\
0.65	0.65\\
nan	nan\\
0.5	0.75\\
0.675	0.825\\
nan	nan\\
0.5	1\\
0.25	1\\
nan	nan\\
0.5	1\\
0.35	0.85\\
nan	nan\\
0.5	1\\
0.75	1\\
nan	nan\\
0.5	1\\
0.675	0.825\\
nan	nan\\
0.25	1\\
0	1\\
nan	nan\\
0.25	1\\
0.35	0.85\\
nan	nan\\
0.25	1\\
0.175	0.825\\
nan	nan\\
0	1\\
0	0.75\\
nan	nan\\
0	1\\
0.175	0.825\\
nan	nan\\
0	0.75\\
0.15	0.65\\
nan	nan\\
0	0.75\\
0.175	0.825\\
nan	nan\\
0.15	0.65\\
0.325	0.675\\
nan	nan\\
0.15	0.65\\
0.175	0.825\\
nan	nan\\
0.325	0.675\\
0.35	0.85\\
nan	nan\\
0.325	0.675\\
0.175	0.825\\
nan	nan\\
0.35	0.85\\
0.175	0.825\\
nan	nan\\
0.75	0.5\\
1	0.5\\
nan	nan\\
0.75	0.5\\
0.65	0.65\\
nan	nan\\
0.75	0.5\\
0.825	0.675\\
nan	nan\\
0.75	0.5\\
0.85	0.35\\
nan	nan\\
0.75	0.5\\
0.675	0.325\\
nan	nan\\
1	0.5\\
1	0.75\\
nan	nan\\
1	0.5\\
0.825	0.675\\
nan	nan\\
1	0.5\\
1	0.25\\
nan	nan\\
1	0.5\\
0.85	0.35\\
nan	nan\\
1	0.75\\
1	1\\
nan	nan\\
1	0.75\\
0.825	0.675\\
nan	nan\\
1	0.75\\
0.85	0.85\\
nan	nan\\
1	1\\
0.75	1\\
nan	nan\\
1	1\\
0.85	0.85\\
nan	nan\\
0.75	1\\
0.85	0.85\\
nan	nan\\
0.75	1\\
0.675	0.825\\
nan	nan\\
0.65	0.65\\
0.825	0.675\\
nan	nan\\
0.65	0.65\\
0.675	0.825\\
nan	nan\\
0.825	0.675\\
0.85	0.85\\
nan	nan\\
0.825	0.675\\
0.675	0.825\\
nan	nan\\
0.85	0.85\\
0.675	0.825\\
nan	nan\\
0	0\\
0.25	0\\
nan	nan\\
0	0\\
0	0.25\\
nan	nan\\
0	0\\
0.15	0.15\\
nan	nan\\
0.25	0\\
0.5	0\\
nan	nan\\
0.25	0\\
0.15	0.15\\
nan	nan\\
0.25	0\\
0.325	0.175\\
nan	nan\\
0.5	0\\
0.5	0.25\\
nan	nan\\
0.5	0\\
0.325	0.175\\
nan	nan\\
0.5	0\\
0.75	0\\
nan	nan\\
0.5	0\\
0.65	0.15\\
nan	nan\\
0.5	0.25\\
0.325	0.175\\
nan	nan\\
0.5	0.25\\
0.35	0.35\\
nan	nan\\
0.5	0.25\\
0.65	0.15\\
nan	nan\\
0.5	0.25\\
0.675	0.325\\
nan	nan\\
0	0.25\\
0.15	0.15\\
nan	nan\\
0	0.25\\
0.175	0.325\\
nan	nan\\
0.15	0.15\\
0.325	0.175\\
nan	nan\\
0.15	0.15\\
0.175	0.325\\
nan	nan\\
0.325	0.175\\
0.35	0.35\\
nan	nan\\
0.325	0.175\\
0.175	0.325\\
nan	nan\\
0.35	0.35\\
0.175	0.325\\
nan	nan\\
0.75	0\\
1	0\\
nan	nan\\
0.75	0\\
0.65	0.15\\
nan	nan\\
0.75	0\\
0.825	0.175\\
nan	nan\\
1	0\\
1	0.25\\
nan	nan\\
1	0\\
0.825	0.175\\
nan	nan\\
1	0.25\\
0.825	0.175\\
nan	nan\\
1	0.25\\
0.85	0.35\\
nan	nan\\
0.65	0.15\\
0.825	0.175\\
nan	nan\\
0.65	0.15\\
0.675	0.325\\
nan	nan\\
0.825	0.175\\
0.85	0.35\\
nan	nan\\
0.825	0.175\\
0.675	0.325\\
nan	nan\\
0.85	0.35\\
0.675	0.325\\
nan	nan\\
};
\end{axis}
\end{tikzpicture}%
    \end{minipage}
        \caption{\textbf{TPFA.} $L^1$ distance between the numerical solution and the steady state (left) on the coarsest mesh (right).}\label{fig:decay}
    \end{figure}

    \subsection{Long time behavior of DDFV schemes}
    
    {We consider a generalization of the test case of \cite{CCHK}. The domain is given by $\Omega=[0,1]\times[0,1]$ and the potential is defined as $V(x_1,x_2)=-x_2$. A family of solutions to \eqref{FPmodel} is given by
    $$
     u_{\varepsilon}(x_1,x_2,t)\ =\ \pi e^{(x_2-\frac{1}{2})} + e^{-(\pi^2+\frac 14)t+\frac{x_2}{2}}\left(\pi\cos(\pi x_2)+\frac{1}{2}\sin(\pi x_2)\right) + \varepsilon\,e^{-\pi^2\Lambda_{11}t}\,\cos(\pi\,x_1)\,.
    $$
    in the case where $\Gamma^D=\emptyset$ and the anisotropy matrix is
    \[
    {\bm \varLambda}=\left(\begin{matrix}
                      \Lambda_{11}&0\\
                      0&1
                     \end{matrix}\right)\,.
    \]
    
    Let us give two comments on this test case. 
    
    First observe that if $\pi^2\Lambda_{11}\ll\pi^2+\frac 14$ and $|\varepsilon| \ll 1$, there are two regimes of decay to equilibrium which are reminiscent of the anisotropy of the equation. Indeed, in the expression of $u_\varepsilon$, the second $x_2$-depending term decaying like $e^{-(\pi^2+\frac 14)t}$ is dominant for small times because $\varepsilon$ is small. Then for larger times the last $x_1$ depending term decaying like $e^{-\Lambda_{11}\pi^2t}$ is greater.

    Second, observe that in the borderline case $\varepsilon =0$, the test case is ``degenerate'' in terms of exponential rate of decay in the following sense. With the same reasoning as before one sees that the minimal decay rate is $\pi^2+\frac 14$ when $\varepsilon =0$ and switches to $\min(\pi^2+\frac 14, \Lambda_{11}\pi^2)$ when $\varepsilon\neq0$ regardless of the size of the perturbation. When $\varepsilon =0$ and $\Lambda_{11}$ is small, any perturbation in the $x_1$ direction (whose first mode in the Fourier decomposition is precisely of the form $e^{-\pi^2\Lambda_{11}t}\,\cos(\pi\,x_1)$)  interferes with the expected decay at rate $\pi^2+\frac 14$. Therefore, unless the symmetries of this particular test case are conserved by the numerical method, one should witness the parasite numerical decay at rate $\pi^2\Lambda_{11}$.

    Let us now describe our numerical tests. We performed numerical simulations on a family of distorted meshes introduced in \cite{bench_FVCA5}, named Kershaw and quadrangular meshes, and on Cartesian meshes. The distorted meshes are represented on Figure~\ref{fig:meshes}. For conciseness we only report here numerical results on Cartesian and distorted quadrangular meshes as the results on Kershaw meshes are qualitatively similar to those of distorted quadrangular meshes. TPFA schemes cannot be used on the former distorted meshes, which motivates the use of the nonlinear DDFV scheme \eqref{schemeDDFV}. Both functions $f$ and $r$ involved in the definition of the scheme are equal to the mean value function, leading to 
    $$
        r^\D(\ut)= \ds\frac{\uk+\ul+\uke+\ule}{4}, \quad \forall {\mathcal D}\in\DD.
    $$
    
    \begin{figure}\centering
     \begin{tabular}{ccc}
      \includegraphics[width = .3\textwidth]{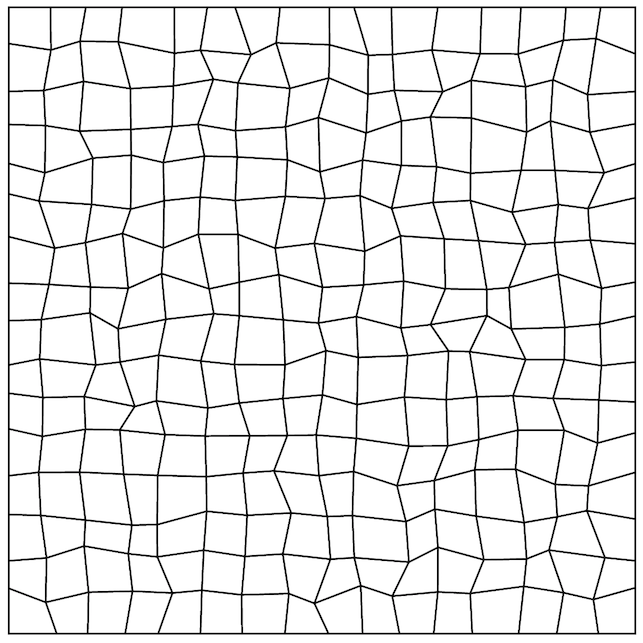}&\quad&
      \includegraphics[width = .3\textwidth]{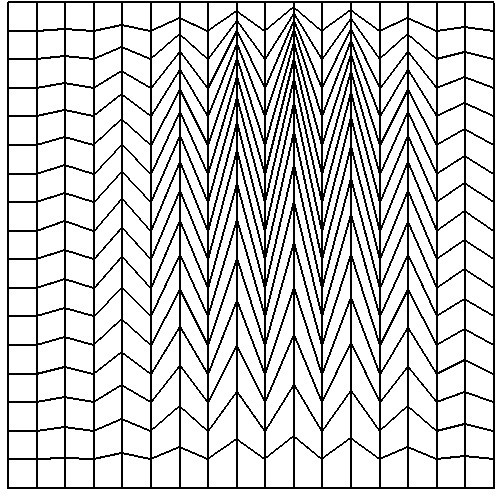}      
     \end{tabular}
     \caption{\textbf{DDFV.} The distorted quandrangular (left) and Kershaw (right) meshes.}\label{fig:meshes}
    \end{figure}

    We start with the test case $u_\varepsilon$ with the choice $\varepsilon =10^{-2}$.  Figure \ref{fig:decayDDFV_epspos}  shows the exponential decay of the relative entropy ${\mathbb E}_{1,\T}$ with respect to time for the different meshes of the sequence of Cartesian and distorted quadrangular meshes. The numerical simulations are performed in both the isotropic case $\Lambda_{11} = 1$ and the anisotropic case $\Lambda_{11} = 0.1$. In both cases the decay rates are well captured by the schemes with better accuracy on the finer meshes (Mesh 4). In the anisotropic case the two regimes of decay are captured by the scheme.
    
     \begin{figure}
     \begin{tabular}{cc}
     $\Lambda_{11} = 1$&$\Lambda_{11} = 0.1$\\[1em]
      \begin{minipage}{0.49\textwidth}\centering
       \scalebox{.83}{  
%
%
\begin{tikzpicture}[scale=0.9]
\definecolor{mycolor1}{rgb}{0.54297,0.22656,0.38281}%
\definecolor{mycolor2}{rgb}{0.00000,0.00000,0.73828}%
\definecolor{mycolor3}{rgb}{0.99609,0.39844,0.11328}%

\begin{axis}[%
tick align=outside,
ytick pos=right,
xtick pos=left,
xmin=0,
xmax=4,
ymode=log,
ymin=1e-20,
ymax=10,
xlabel={$t$},
ylabel={Discrete relative entropy},
axis background/.style={fill=white},
xmajorgrids,
ymajorgrids,
legend style={legend cell align=left, align=left, draw=white!80!black}
]
\addplot [color=red, line width=2.0pt, mark size=2pt, mark=triangle, mark options={solid, rotate=180, red}]
  table[row sep=crcr]{%
   0.032     0.7818882\\
   0.352     0.0035724\\
   0.672     0.0000165\\
   0.992     7.592E-08\\
   1.312     3.495E-10\\
   1.632     1.612E-12\\
   1.952     7.409E-15\\
   2.272     5.725E-17\\
   2.592     4.163E-17\\
   2.912     1.388E-17\\
   3.232     2.429E-17\\
   3.552     1.041E-17\\
   3.872     6.939E-17\\
   };
\addlegendentry{Mesh 1}

\addplot [color=mycolor1, line width=1.0pt, mark size=2pt, mark=+, mark options={solid, mycolor1}]
  table[row sep=crcr]{%
    0.008     1.2048535\\
   0.328     0.0023992\\
   0.648     0.0000051\\
   0.968     1.081E-08\\
   1.288     2.287E-11\\
   1.608     4.837E-14\\
   1.928     5.291E-17\\
   2.248     7.373E-18\\
   2.568     3.469E-18\\
   2.888     9.541E-18\\
   3.208     4.597E-17\\
   3.528     6.505E-18\\
   };
\addlegendentry{Mesh 2}

\addplot [color=black, line width=1.0pt, mark size=2pt, mark=triangle, mark options={solid, rotate=270, black}]
  table[row sep=crcr]{%
  0.002     1.3926933\\
   0.322     0.0021221\\
   0.642     0.0000036\\
   0.962     5.970E-09\\
   1.282     9.984E-12\\
   1.602     1.671E-14\\
   1.922     3.350E-17\\
   2.242     1.735E-18\\
   2.562     3.036E-18\\
   2.882     4.987E-18\\
   3.202     3.220E-17\\
   3.522     4.467E-17\\
   3.842     2.168E-19\\
};
\addlegendentry{Mesh 3}

\addplot [color=mycolor2, line width=1.0pt, mark size=2pt, mark=o, mark options={solid, mycolor2}]
  table[row sep=crcr]{%
   0.0005    1.4521096\\
   0.3205    0.0020543\\
   0.6405    0.0000032\\
   0.9605    5.104E-09\\
   1.2805    8.023E-12\\
   1.6005    1.261E-14\\
   1.9205    1.721E-17\\
   2.2405    4.228E-18\\
   2.5605    5.150E-18\\
   2.8805    2.738E-18\\
   3.2005    8.132E-19\\
   3.5205    1.301E-17\\
   3.8405    2.103E-17\\
};
\addlegendentry{Mesh 4}

\addplot [color=mycolor3, dashed, line width=2.0pt]
  table[row sep=crcr]{%
0.0005	0.989931426512908\\
0.3205	0.00152361965801251\\
0.6405	2.34502794851102e-06\\
0.9605	3.60927089013225e-09\\
1.2805	5.55508788994497e-12\\
1.6005	8.54992667615546e-15\\
1.9205	1.31593320602457e-17\\
2.2405	2.02537433162731e-20\\
2.5605	3.11728677750091e-23\\
2.8805	4.79786709125244e-26\\
};
\addlegendentry{$ t \mapsto e^{-2(\pi^2+1/4)t}$}

\end{axis}
\end{tikzpicture}%
}
      \end{minipage}&
      \begin{minipage}{0.45\textwidth}\centering
       \scalebox{.83}{  
%
%
\begin{tikzpicture}[scale=0.9]
\definecolor{mycolor1}{rgb}{0.54297,0.22656,0.38281}%
\definecolor{mycolor2}{rgb}{0.00000,0.00000,0.73828}%
\definecolor{mycolor3}{rgb}{0.99609,0.39844,0.11328}%
\definecolor{mycolor4}{rgb}{0,0.59,0}%

\begin{axis}[%
tick align=outside,
ytick pos=right,
xtick pos=left,
xmin=0,
xmax=12,
ymode=log,
ymin=1e-16,
ymax=10,
xlabel={$t$},
axis background/.style={fill=white},
xmajorgrids,
ymajorgrids,
legend style={legend cell align=left, align=left, draw=white!80!black}
]
\addplot [color=red, line width=2.0pt, mark size=2pt, mark=triangle, mark options={solid, rotate=180, red}]
  table[row sep=crcr]{%
   0.032     0.7818918\\
   0.352     0.0035764\\
   0.672     0.0000187\\
   0.992     0.0000013\\
   1.312     0.0000007\\
   1.632     0.0000004\\
   1.952     0.0000002\\
   2.272     0.0000001\\
   2.592     6.323E-08\\
   2.912     3.499E-08\\
   3.232     1.936E-08\\
   3.552     1.071E-08\\
   3.872     5.926E-09\\
   4.192     3.279E-09\\
   4.512     1.815E-09\\
   4.832     1.004E-09\\
   5.152     5.556E-10\\
   5.472     3.075E-10\\
   5.792     1.701E-10\\
   6.112     9.416E-11\\
   6.432     5.211E-11\\
   6.752     2.884E-11\\
   7.072     1.596E-11\\
   7.392     8.833E-12\\
   7.712     4.889E-12\\
   8.032     2.706E-12\\
   8.352     1.497E-12\\
   8.672     8.289E-13\\
   8.992     4.588E-13\\
   9.312     2.540E-13\\
   9.632     1.406E-13\\
   9.952     7.780E-14\\
   10.272    4.306E-14\\
   10.592    2.389E-14\\
   10.912    1.320E-14\\
   11.232    7.289E-15\\
   11.552    4.002E-15\\
   11.872    2.293E-15\\
};
\addlegendentry{Mesh 1}

\addplot [color=mycolor1, line width=1.0pt, mark size=2pt, mark=+, mark options={solid, mycolor1}]
  table[row sep=crcr]{%
   0.008     1.2048555\\
   0.328     0.0024032\\
   0.648     0.0000073\\
   0.968     0.0000012\\
   1.288     0.0000006\\
   1.608     0.0000003\\
   1.928     0.0000002\\
   2.248     9.703E-08\\
   2.568     5.212E-08\\
   2.888     2.800E-08\\
   3.208     1.504E-08\\
   3.528     8.080E-09\\
   3.848     4.340E-09\\
   4.168     2.331E-09\\
   4.488     1.252E-09\\
   4.808     6.728E-10\\
   5.128     3.614E-10\\
   5.448     1.941E-10\\
   5.768     1.043E-10\\
   6.088     5.602E-11\\
   6.408     3.009E-11\\
   6.728     1.617E-11\\
   7.048     8.684E-12\\
   7.368     4.665E-12\\
   7.688     2.506E-12\\
   8.008     1.346E-12\\
   8.328     7.231E-13\\
   8.648     3.884E-13\\
   8.968     2.087E-13\\
   9.288     1.121E-13\\
   9.608     6.024E-14\\
   9.928     3.233E-14\\
   10.248    1.739E-14\\
   10.568    9.324E-15\\
   10.888    5.025E-15\\
   11.208    2.723E-15\\
   11.528    1.464E-15\\
   11.848    7.637E-16\\
};
\addlegendentry{Mesh 2}

\addplot [color=black, line width=1.0pt, mark size=2pt, mark=triangle, mark options={solid, rotate=270, black}]
  table[row sep=crcr]{%
    0.002     1.3926942\\
   0.322     0.0021261\\
   0.642     0.0000057\\
   0.962     0.0000012\\
   1.282     0.0000006\\
   1.602     0.0000003\\
   1.922     0.0000002\\
   2.242     9.304E-08\\
   2.562     4.960E-08\\
   2.882     2.644E-08\\
   3.202     1.409E-08\\
   3.522     7.514E-09\\
   3.842     4.005E-09\\
   4.162     2.135E-09\\
   4.482     1.138E-09\\
   4.802     6.068E-10\\
   5.122     3.235E-10\\
   5.442     1.724E-10\\
   5.762     9.192E-11\\
   6.082     4.900E-11\\
   6.402     2.612E-11\\
   6.722     1.392E-11\\
   7.042     7.423E-12\\
   7.362     3.957E-12\\
   7.682     2.109E-12\\
   8.002     1.124E-12\\
   8.322     5.995E-13\\
   8.642     3.195E-13\\
   8.962     1.703E-13\\
   9.282     9.081E-14\\
   9.602     4.841E-14\\
   9.922     2.580E-14\\
   10.242    1.376E-14\\
   10.562    7.325E-15\\
   10.882    3.911E-15\\
   11.202    2.077E-15\\
   11.522    1.093E-15\\
   11.842    5.920E-16\\
  };
\addlegendentry{Mesh 3}

\addplot [color=mycolor2, line width=1.0pt, mark size=2pt, mark=o, mark options={solid, mycolor2}]
  table[row sep=crcr]{%
 0.0005    1.4521101\\
   0.3205    0.0020583\\
   0.6405    0.0000054\\
   0.9605    0.0000012\\
   1.2805    0.0000006\\
   1.6005    0.0000003\\
   1.9205    0.0000002\\
   2.2405    9.207E-08\\
   2.5605    4.898E-08\\
   2.8805    2.606E-08\\
   3.2005    1.387E-08\\
   3.5205    7.378E-09\\
   3.8405    3.925E-09\\
   4.1605    2.088E-09\\
   4.4805    1.111E-09\\
   4.8005    5.912E-10\\
   5.1205    3.146E-10\\
   5.4405    1.674E-10\\
   5.7605    8.904E-11\\
   6.0805    4.738E-11\\
   6.4005    2.521E-11\\
   6.7205    1.341E-11\\
   7.0405    7.136E-12\\
   7.3605    3.797E-12\\
   7.6805    2.020E-12\\
   8.0005    1.075E-12\\
   8.3205    5.718E-13\\
   8.6405    3.042E-13\\
   8.9605    1.619E-13\\
   9.2805    8.612E-14\\
   9.6005    4.582E-14\\
   9.9205    2.438E-14\\
   10.2405   1.298E-14\\
   10.5605   6.904E-15\\
   10.8805   3.676E-15\\
   11.2005   1.955E-15\\
   11.5205   1.038E-15\\
   11.8405   5.508E-16\\
};
\addlegendentry{Mesh 4}

\addplot [color=mycolor3, dashed, line width=2.0pt]
  table[row sep=crcr]{%
0.0005	0.989931426512908\\
0.3205	0.00152361965801251\\
0.6405	2.34502794851102e-06\\
0.9605	3.60927089013225e-09\\
1.2805	5.55508788994497e-12\\
1.6005	8.54992667615546e-15\\
1.9205	1.31593320602457e-17\\
2.2405	2.02537433162731e-20\\
2.5605	3.11728677750091e-23\\
2.8805	4.79786709125244e-26\\
};
\addlegendentry{$ t \mapsto e^{-2(\pi^2+1/4)t}$}

\addplot [color=mycolor4, dash dot, line width=2.0pt]
  table[row sep=crcr]{%
   0.9605    0.0000002\\
   1.2805    7.985E-08\\
   1.6005    4.246E-08\\
   1.9205    2.257E-08\\
   2.2405    1.200E-08\\
   2.5605    6.382E-09\\
   2.8805    3.394E-09\\
   3.2005    1.804E-09\\
   3.5205    9.594E-10\\
   3.8405    5.101E-10\\
   4.1605    2.712E-10\\
   4.4805    1.442E-10\\
   4.8005    7.668E-11\\
   5.1205    4.077E-11\\
   5.4405    2.168E-11\\
   5.7605    1.153E-11\\
   6.0805    6.129E-12\\
   6.4005    3.259E-12\\
   6.7205    1.733E-12\\
   7.0405    9.214E-13\\
   7.3605    4.899E-13\\
   7.6805    2.605E-13\\
   8.0005    1.385E-13\\
   8.3205    7.365E-14\\
   8.6405    3.916E-14\\
   8.9605    2.082E-14\\
   9.2805    1.107E-14\\
   9.6005    5.886E-15\\
   9.9205    3.130E-15\\
   10.2405   1.664E-15\\
   10.5605   8.849E-16\\
   10.8805   4.705E-16\\
   11.2005   2.502E-16\\
   11.5205   1.330E-16\\
   11.8405   7.073E-17\\
};
\addlegendentry{$ t \mapsto 0.01\varepsilon^2e^{-2\pi^2\Lambda_{11}t}$}

\end{axis}
\end{tikzpicture}%
}
      \end{minipage}\\
      \begin{minipage}{0.49\textwidth}\centering
       \scalebox{.83}{  
%
%
\begin{tikzpicture}[scale=0.9]
\definecolor{mycolor1}{rgb}{0.54297,0.22656,0.38281}%
\definecolor{mycolor2}{rgb}{0.00000,0.00000,0.73828}%
\definecolor{mycolor3}{rgb}{0.99609,0.39844,0.11328}%

\begin{axis}[%
tick align=outside,
ytick pos=right,
xtick pos=left,
xmin=0,
xmax=4,
ymode=log,
ymin=1e-23,
ymax=10,
xlabel={$t$},
ylabel={Discrete relative entropy},
axis background/.style={fill=white},
xmajorgrids,
ymajorgrids,
legend style={legend cell align=left, align=left, draw=white!80!black}
]
\addplot [color=red, line width=2.0pt, mark size=2pt, mark=triangle, mark options={solid, rotate=180, red}]
  table[row sep=crcr]{%
  0.032     0.7865425\\
   0.352     0.0036698\\
   0.672     0.0000172\\
   0.992     7.952E-08\\
   1.312     3.686E-10\\
   1.632     1.710E-12\\
   1.952     7.898E-15\\
   2.272     2.634E-17\\
   2.592     1.418E-17\\
   2.912     1.308E-17\\
   3.232     8.115E-18\\
   3.552     5.981E-18\\
   3.872     2.163E-17\\
   4.192     1.246E-17\\
   4.512     5.085E-18\\
   4.832     2.077E-18\\
   5.152     1.252E-18\\
   5.472     1.580E-17\\
   5.792     2.971E-18\\
   6.112     0.       \\
   6.432     4.624E-17\\
   6.752     8.135E-17\\
   7.072     8.369E-17\\
   7.392     5.898E-17\\
   7.712     4.457E-17\\
   8.032     6.596E-17\\
   8.352     7.624E-17\\
   8.672     2.166E-17\\
   8.992     8.710E-17\\
   9.312     5.902E-17\\
   9.632     2.091E-17\\
   9.952     7.022E-17\\
   10.272    9.144E-18\\
   10.592    3.939E-17\\
   10.912    1.930E-17\\
   11.232    1.159E-17\\
   11.552    4.145E-17\\
   11.872    8.887E-18\\
};
\addlegendentry{Mesh 1}

\addplot [color=mycolor1, line width=1.0pt, mark size=2pt, mark=+, mark options={solid, mycolor1}]
  table[row sep=crcr]{%
   0.008     1.2056398\\
   0.328     0.0024133\\
   0.648     0.0000052\\
   0.968     1.099E-08\\
   1.288     2.338E-11\\
   1.608     4.978E-14\\
   1.928     7.061E-17\\
   2.248     4.042E-18\\
   2.568     3.409E-17\\
   2.888     4.109E-17\\
   3.208     1.820E-17\\
   3.528     2.450E-17\\
   3.848     3.399E-17\\
   4.168     9.790E-18\\
   4.488     2.767E-18\\
   4.808     8.611E-18\\
   5.128     1.887E-18\\
   5.448     1.475E-18\\
   5.768     5.137E-18\\
   6.088     1.162E-17\\
   6.408     5.752E-18\\
   6.728     1.288E-17\\
   7.048     5.869E-18\\
   7.368     3.175E-17\\
   7.688     2.194E-17\\
   8.008     1.889E-17\\
   8.328     1.175E-17\\
   8.648     8.021E-18\\
   8.968     1.174E-17\\
   9.288     3.383E-18\\
   9.608     1.946E-17\\
   9.928     4.767E-18\\
   10.248    1.556E-17\\
   10.568    2.098E-18\\
   10.888    2.839E-18\\
   11.208    4.619E-18\\
   11.528    4.833E-18\\
   11.848    2.991E-17\\
};
\addlegendentry{Mesh 2}

\addplot [color=black, line width=1.0pt, mark size=2pt, mark=triangle, mark options={solid, rotate=270, black}]
  table[row sep=crcr]{%
  0.002     1.3928665\\
   0.322     0.0021261\\
   0.642     0.0000036\\
   0.962     6.001E-09\\
   1.282     1.005E-11\\
   1.602     1.684E-14\\
   1.922     1.884E-17\\
   2.242     3.997E-18\\
   2.562     5.202E-18\\
   2.882     4.047E-19\\
   3.202     2.252E-18\\
   3.522     8.134E-19\\
   3.842     1.099E-17\\
   4.162     6.601E-18\\
   4.482     1.003E-17\\
   4.802     1.018E-17\\
   5.122     1.696E-17\\
   5.442     1.366E-17\\
   5.762     1.369E-17\\
   6.082     4.325E-18\\
   6.402     1.006E-17\\
   6.722     1.735E-17\\
   7.042     1.328E-17\\
   7.362     2.059E-17\\
   7.682     2.146E-17\\
   8.002     1.632E-17\\
   8.322     1.295E-17\\
   8.642     1.896E-17\\
   8.962     1.196E-17\\
   9.282     1.602E-17\\
   9.602     2.096E-17\\
   9.922     1.520E-17\\
   10.242    2.510E-17\\
   10.562    1.799E-17\\
   10.882    1.721E-17\\
   11.202    1.541E-17\\
   11.522    1.085E-17\\
   11.842    1.799E-17\\
  };
\addlegendentry{Mesh 3}

\addplot [color=mycolor2, line width=1.0pt, mark size=2pt, mark=o, mark options={solid, mycolor2}]
  table[row sep=crcr]{%
  0.0005    1.452132 \\
   0.3205    0.0020551\\
   0.6405    0.0000032\\
   0.9605    5.110E-09\\
   1.2805    8.036E-12\\
   1.6005    1.263E-14\\
   1.9205    2.014E-17\\
   2.2405    1.054E-18\\
   2.5605    2.285E-18\\
   2.8805    4.137E-18\\
   3.2005    5.917E-18\\
   3.5205    3.310E-18\\
   3.8405    3.510E-18\\
   4.1605    3.404E-18\\
   4.4805    4.828E-18\\
   4.8005    6.790E-18\\
   5.1205    3.188E-18\\
   5.4405    5.615E-19\\
   5.7605    3.884E-18\\
   6.0805    1.685E-18\\
   6.4005    4.284E-18\\
   6.7205    2.690E-18\\
   7.0405    4.279E-18\\
   7.3605    2.342E-18\\
   7.6805    6.255E-18\\
   8.0005    4.433E-18\\
   8.3205    3.407E-18\\
   8.6405    6.353E-18\\
   8.9605    3.334E-18\\
   9.2805    3.635E-18\\
   9.6005    3.834E-18\\
   9.9205    5.151E-18\\
   10.2405   1.132E-18\\
   10.5605   2.588E-18\\
   10.8805   4.335E-18\\
   11.2005   4.257E-18\\
   11.5205   4.230E-18\\
   11.8405   4.329E-18\\
  };
\addlegendentry{Mesh 4}

\addplot [color=mycolor3, dashed, line width=2.0pt]
  table[row sep=crcr]{%
0.0005	0.989931426512908\\
0.3205	0.00152361965801251\\
0.6405	2.34502794851102e-06\\
0.9605	3.60927089013225e-09\\
1.2805	5.55508788994497e-12\\
1.6005	8.54992667615546e-15\\
1.9205	1.31593320602457e-17\\
2.2405	2.02537433162731e-20\\
2.5605	3.11728677750091e-23\\
2.8805	4.79786709125244e-26\\
};
\addlegendentry{$ t \mapsto e^{-2(\pi^2+1/4)t}$}

\end{axis}
\end{tikzpicture}%
}
      \end{minipage}&
      \begin{minipage}{0.45\textwidth}\centering
       \scalebox{.83}{  
%
%
\begin{tikzpicture}[scale=0.9]
\definecolor{mycolor1}{rgb}{0.54297,0.22656,0.38281}%
\definecolor{mycolor2}{rgb}{0.00000,0.00000,0.73828}%
\definecolor{mycolor3}{rgb}{0.99609,0.39844,0.11328}%
\definecolor{mycolor4}{rgb}{0,0.59,0}%

\begin{axis}[%
tick align=outside,
ytick pos=right,
xtick pos=left,
xmin=0,
xmax=12,
ymode=log,
ymin=1e-16,
ymax=10,
xlabel={$t$},
axis background/.style={fill=white},
xmajorgrids,
ymajorgrids,
legend style={legend cell align=left, align=left, draw=white!80!black}
]
\addplot [color=red, line width=2.0pt, mark size=2pt, mark=triangle, mark options={solid, rotate=180, red}]
  table[row sep=crcr]{%
   0.032     0.7866276\\
   0.352     0.0037181\\
   0.672     0.0000216\\
   0.992     0.0000016\\
   1.312     0.0000008\\
   1.632     0.0000005\\
   1.952     0.0000003\\
   2.272     0.0000001\\
   2.592     7.719E-08\\
   2.912     4.285E-08\\
   3.232     2.378E-08\\
   3.552     1.320E-08\\
   3.872     7.328E-09\\
   4.192     4.068E-09\\
   4.512     2.258E-09\\
   4.832     1.254E-09\\
   5.152     6.959E-10\\
   5.472     3.863E-10\\
   5.792     2.144E-10\\
   6.112     1.190E-10\\
   6.432     6.608E-11\\
   6.752     3.669E-11\\
   7.072     2.037E-11\\
   7.392     1.131E-11\\
   7.712     6.276E-12\\
   8.032     3.484E-12\\
   8.352     1.934E-12\\
   8.672     1.074E-12\\
   8.992     5.962E-13\\
   9.312     3.310E-13\\
   9.632     1.838E-13\\
   9.952     1.020E-13\\
   10.272    5.671E-14\\
   10.592    3.147E-14\\
   10.912    1.744E-14\\
   11.232    9.676E-15\\
   11.552    5.415E-15\\
   11.872    3.054E-15\\
  };
\addlegendentry{Mesh 1}

\addplot [color=mycolor1, line width=1.0pt, mark size=2pt, mark=+, mark options={solid, mycolor1}]
  table[row sep=crcr]{%
   0.008     1.2056494\\
   0.328     0.0024216\\
   0.648     0.0000084\\
   0.968     0.0000017\\
   1.288     0.0000009\\
   1.608     0.0000005\\
   1.928     0.0000003\\
   2.248     0.0000001\\
   2.568     7.757E-08\\
   2.888     4.169E-08\\
   3.208     2.241E-08\\
   3.528     1.204E-08\\
   3.848     6.472E-09\\
   4.168     3.478E-09\\
   4.488     1.869E-09\\
   4.808     1.005E-09\\
   5.128     5.400E-10\\
   5.448     2.902E-10\\
   5.768     1.560E-10\\
   6.088     8.384E-11\\
   6.408     4.506E-11\\
   6.728     2.422E-11\\
   7.048     1.302E-11\\
   7.368     6.997E-12\\
   7.688     3.761E-12\\
   8.008     2.021E-12\\
   8.328     1.086E-12\\
   8.648     5.840E-13\\
   8.968     3.139E-13\\
   9.288     1.687E-13\\
   9.608     9.068E-14\\
   9.928     4.872E-14\\
   10.248    2.620E-14\\
   10.568    1.409E-14\\
   10.888    7.595E-15\\
   11.208    4.084E-15\\
   11.528    2.188E-15\\
   11.848    1.154E-15\\
  };
\addlegendentry{Mesh 2}

\addplot [color=black, line width=1.0pt, mark size=2pt, mark=triangle, mark options={solid, rotate=270, black}]
  table[row sep=crcr]{%
   0.002     1.3928683\\
   0.322     0.0021303\\
   0.642     0.0000057\\
   0.962     0.0000011\\
   1.282     0.0000006\\
   1.602     0.0000003\\
   1.922     0.0000002\\
   2.242     9.077E-08\\
   2.562     4.840E-08\\
   2.882     2.581E-08\\
   3.202     1.376E-08\\
   3.522     7.336E-09\\
   3.842     3.911E-09\\
   4.162     2.086E-09\\
   4.482     1.112E-09\\
   4.802     5.929E-10\\
   5.122     3.161E-10\\
   5.442     1.685E-10\\
   5.762     8.986E-11\\
   6.082     4.791E-11\\
   6.402     2.555E-11\\
   6.722     1.362E-11\\
   7.042     7.263E-12\\
   7.362     3.872E-12\\
   7.682     2.065E-12\\
   8.002     1.101E-12\\
   8.322     5.870E-13\\
   8.642     3.129E-13\\
   8.962     1.669E-13\\
   9.282     8.896E-14\\
   9.602     4.742E-14\\
   9.922     2.529E-14\\
   10.242    1.349E-14\\
   10.562    7.221E-15\\
   10.882    3.843E-15\\
   11.202    2.052E-15\\
   11.522    1.089E-15\\
   11.842    5.814E-16\\
  };
\addlegendentry{Mesh 3}

\addplot [color=mycolor2, line width=1.0pt, mark size=2pt, mark=o, mark options={solid, mycolor2}]
  table[row sep=crcr]{%
  0.0005    1.4521326\\
   0.3205    0.0020593\\
   0.6405    0.0000054\\
   0.9605    0.0000012\\
   1.2805    0.0000006\\
   1.6005    0.0000003\\
   1.9205    0.0000002\\
   2.2405    9.258E-08\\
   2.5605    4.926E-08\\
   2.8805    2.621E-08\\
   3.2005    1.395E-08\\
   3.5205    7.420E-09\\
   3.8405    3.948E-09\\
   4.1605    2.101E-09\\
   4.4805    1.118E-09\\
   4.8005    5.947E-10\\
   5.1205    3.164E-10\\
   5.4405    1.684E-10\\
   5.7605    8.958E-11\\
   6.0805    4.766E-11\\
   6.4005    2.536E-11\\
   6.7205    1.349E-11\\
   7.0405    7.180E-12\\
   7.3605    3.820E-12\\
   7.6805    2.033E-12\\
   8.0005    1.082E-12\\
   8.3205    5.755E-13\\
   8.6405    3.062E-13\\
   8.9605    1.629E-13\\
   9.2805    8.669E-14\\
   9.6005    4.612E-14\\
   9.9205    2.454E-14\\
   10.2405   1.305E-14\\
   10.5605   6.946E-15\\
   10.8805   3.705E-15\\
   11.2005   1.972E-15\\
   11.5205   1.051E-15\\
   11.8405   5.591E-16\\
  };
\addlegendentry{Mesh 4}

\addplot [color=mycolor3, dashed, line width=2.0pt]
  table[row sep=crcr]{%
0.0005	0.989931426512908\\
0.3205	0.00152361965801251\\
0.6405	2.34502794851102e-06\\
0.9605	3.60927089013225e-09\\
1.2805	5.55508788994497e-12\\
1.6005	8.54992667615546e-15\\
1.9205	1.31593320602457e-17\\
2.2405	2.02537433162731e-20\\
2.5605	3.11728677750091e-23\\
2.8805	4.79786709125244e-26\\
};
\addlegendentry{$ t \mapsto e^{-2(\pi^2+1/4)t}$}

\addplot [color=mycolor4, dash dot, line width=2.0pt]
  table[row sep=crcr]{%
   0.9605    0.0000002\\
   1.2805    7.985E-08\\
   1.6005    4.246E-08\\
   1.9205    2.257E-08\\
   2.2405    1.200E-08\\
   2.5605    6.382E-09\\
   2.8805    3.394E-09\\
   3.2005    1.804E-09\\
   3.5205    9.594E-10\\
   3.8405    5.101E-10\\
   4.1605    2.712E-10\\
   4.4805    1.442E-10\\
   4.8005    7.668E-11\\
   5.1205    4.077E-11\\
   5.4405    2.168E-11\\
   5.7605    1.153E-11\\
   6.0805    6.129E-12\\
   6.4005    3.259E-12\\
   6.7205    1.733E-12\\
   7.0405    9.214E-13\\
   7.3605    4.899E-13\\
   7.6805    2.605E-13\\
   8.0005    1.385E-13\\
   8.3205    7.365E-14\\
   8.6405    3.916E-14\\
   8.9605    2.082E-14\\
   9.2805    1.107E-14\\
   9.6005    5.886E-15\\
   9.9205    3.130E-15\\
   10.2405   1.664E-15\\
   10.5605   8.849E-16\\
   10.8805   4.705E-16\\
   11.2005   2.502E-16\\
   11.5205   1.330E-16\\
   11.8405   7.073E-17\\
};
\addlegendentry{$ t \mapsto 0.01\varepsilon^2e^{-2\pi^2\Lambda_{11}t}$}

\end{axis}
\end{tikzpicture}%
}
      \end{minipage}
     \end{tabular}
     \caption{\textbf{DDFV.} Exponential decay of the relative entropy on a sequence of Cartesian (top row) and quadrangular (bottom row) meshes mesh with $\varepsilon=10^{-2}$ and $\Lambda_{11} = 1$ (left column) or $\Lambda_{11} = 0.1$ (right column).}\label{fig:decayDDFV_epspos}
    \end{figure}
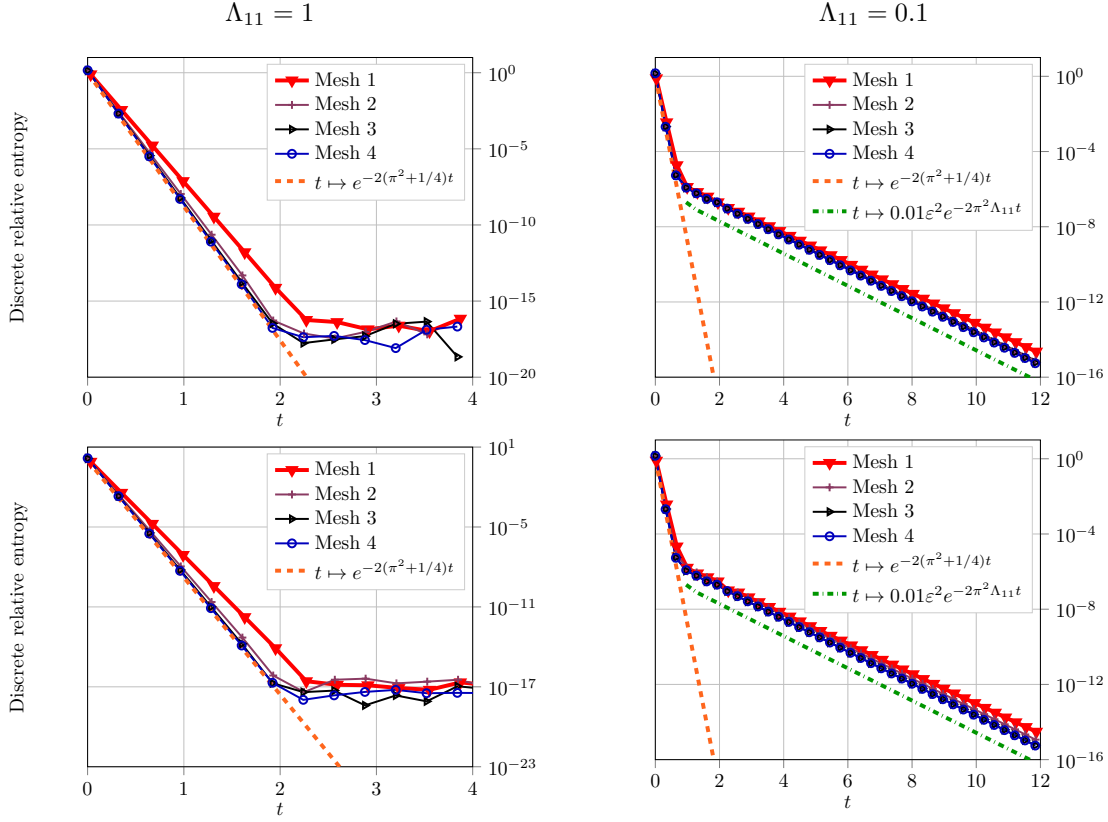

    Then, on Figure~\ref{fig:decayDDFV_epsnull}, we deal with the degenerate test case $\varepsilon = 0$ and $\Lambda_{11} = 0.1$. Once again we plot the evolution of the relative entropy ${\mathbb E}_{1,\T}$ with respect to time for the different meshes of the sequence of Cartesian and distorted quadrangular meshes. This time, the behavior is  different. While the Cartesian mesh is as accurate as in the previous test case, the quadrangular mesh captures a parasite decay starting at around $t =1$. One observes that this effect is mitigated when the mesh is refined. This is in accordance with the above discussion on the ``degeneracy'' of the case $\varepsilon=0$. Indeed, the Cartesian mesh preserves the symmetries of the problem and therefore the discretization does not introduce any perturbation in the $x_1$ direction. On the contrary the distorted mesh does not preserve the symmetry and a  perturbation is unavoidably introduced in the $x_1$ direction. Of course it tends to vanish as the mesh size goes to $0$. In order to further confirm this analysis, one can estimate numerically the ratio between the two experimental rates of decay  (the slopes for small and large times) for the right plot of Figure~\ref{fig:decayDDFV_epsnull}. We find a ratio of $9.49$ for the distoted quandragular mesh, and the ratio of $10.32$ for the same simulation on a Kershaw mesh. The exact ratio is  $(\pi^2+\frac 14)/(\Lambda_{11}\pi^2) \approx 10.25$.

      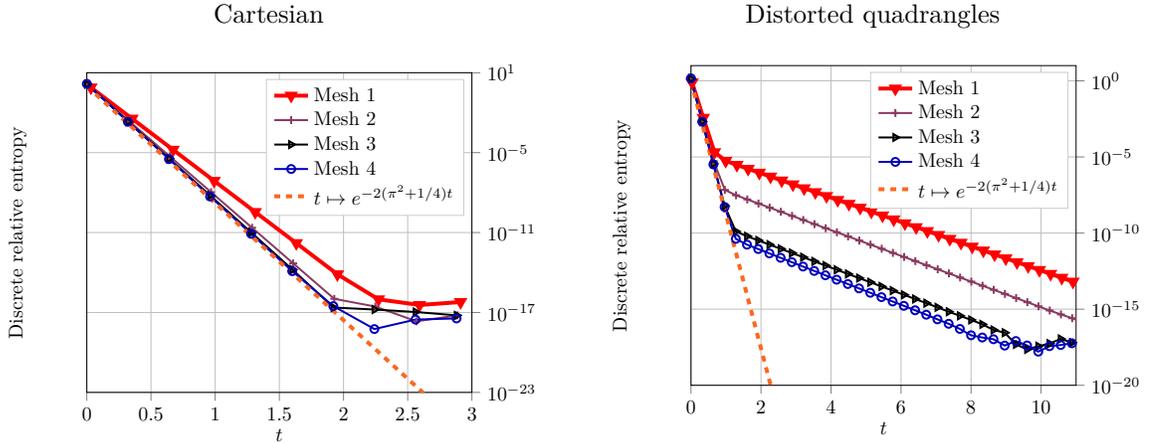
\begin{figure}
     \begin{tabular}{cc}
     Cartesian&Distorted quadrangles\\[1em]
      \begin{minipage}{0.49\textwidth}\centering
       \scalebox{.83}{  
%
%
\begin{tikzpicture}[scale=0.9]
\definecolor{mycolor1}{rgb}{0.54297,0.22656,0.38281}%
\definecolor{mycolor2}{rgb}{0.00000,0.00000,0.73828}%
\definecolor{mycolor3}{rgb}{0.99609,0.39844,0.11328}%

\begin{axis}[%
tick align=outside,
ytick pos=right,
xtick pos=left,
xmin=0,
xmax=3,
ymode=log,
ymin=1e-23,
ymax=10,
xlabel={$t$},
ylabel={Discrete relative entropy},
axis background/.style={fill=white},
xmajorgrids,
ymajorgrids,
legend style={legend cell align=left, align=left, draw=white!80!black}
]
\addplot [color=red, line width=2.0pt, mark size=2pt, mark=triangle, mark options={solid, rotate=180, red}]
  table[row sep=crcr]{%
0.032	 0.78188245998601724\\
0.352	 3.5723690390783089E-003\\
0.672	 1.6518392900075410E-005\\
0.992	 7.5915695972594843E-008\\
1.312	 3.4947013859143450E-010\\
1.632	 1.6122641416371764E-012\\
1.952	7.3864525607092446E-015\\
2.272	 9.7144514654701197E-017\\
2.592	 3.6429192995512949E-017\\
2.912	6.0715321659188248E-017\\
};
\addlegendentry{Mesh 1}

\addplot [color=mycolor1, line width=1.0pt, mark size=2pt, mark=+, mark options={solid, mycolor1}]
  table[row sep=crcr]{%
0.008	 1.2048407259069105\\
0.328	 2.3991645168256488E-003\\
0.648	 5.1073698924815163E-006\\
0.968	 1.0808537387590328E-008\\
1.288	 2.2871834200882679E-011\\
1.608	4.8497664217883596E-014\\
1.928	1.0581813203458523E-016\\
2.248	2.7755575615628914E-017\\
2.568	 2.1684043449710089E-018\\
2.888	6.0715321659188248E-018\\
};
\addlegendentry{Mesh 2}

\addplot [color=black, line width=1.0pt, mark size=2pt, mark=triangle, mark options={solid, rotate=270, black}]
  table[row sep=crcr]{%
0.002	 1.3926683133856972\\
0.322	 2.1220719141431818E-003\\
0.642	 3.5691069419668169E-006\\
0.962	 5.9701852103515030E-009\\
1.282	 9.9843550311132234E-012\\
1.602	 1.6732817388620536E-014\\
1.922	 2.3527187142935446E-017\\
2.242	 1.6588293239028218E-017\\
2.562	 1.0842021724855044E-017\\
2.882	 6.0715321659188248E-018\\
};
\addlegendentry{Mesh 3}

\addplot [color=mycolor2, line width=1.0pt, mark size=2pt, mark=o, mark options={solid, mycolor2}]
  table[row sep=crcr]{%
0.0005	  1.4520615568568647\\
0.3205	  2.0542869404498463E-003\\
0.6405	 3.2467601265071368E-006 \\
0.9605	  5.1042886559961569E-009\\
1.2805	 8.0227620065983357E-012 \\
1.6005	  1.2619435661373468E-014\\
1.9205	  3.0032400177848473E-017\\
2.2405	 5.6920614055488983E-019\\
2.5605	 2.9273458657108620E-018\\
2.8805	 3.5778671692021646E-018\\
};
\addlegendentry{Mesh 4}

\addplot [color=mycolor3, dashed, line width=2.0pt]
  table[row sep=crcr]{%
0.0005	0.989931426512908\\
0.3205	0.00152361965801251\\
0.6405	2.34502794851102e-06\\
0.9605	3.60927089013225e-09\\
1.2805	5.55508788994497e-12\\
1.6005	8.54992667615546e-15\\
1.9205	1.31593320602457e-17\\
2.2405	2.02537433162731e-20\\
2.5605	3.11728677750091e-23\\
2.8805	4.79786709125244e-26\\
};
\addlegendentry{$ t \mapsto e^{-2(\pi^2+1/4)t}$}

\end{axis}
\end{tikzpicture}%
}
      \end{minipage}&
      \begin{minipage}{0.45\textwidth}\centering
       \scalebox{.83}{  
%
%
\begin{tikzpicture}[scale=0.9]
\definecolor{mycolor1}{rgb}{0.54297,0.22656,0.38281}%
\definecolor{mycolor2}{rgb}{0.00000,0.00000,0.73828}%
\definecolor{mycolor3}{rgb}{0.99609,0.39844,0.11328}%

\begin{axis}[%
tick align=outside,
ytick pos=right,
xtick pos=left,
xmin=0,
xmax=11,
ymode=log,
ymin=1e-20,
ymax=10,
xlabel={$t$},
ylabel={Discrete relative entropy},
axis background/.style={fill=white},
xmajorgrids,
ymajorgrids,
legend style={legend cell align=left, align=left, draw=white!80!black}
]
\addplot [color=red, line width=2.0pt, mark size=2pt, mark=triangle, mark options={solid, rotate=180, red}]
  table[row sep=crcr]{%
0.032	0.78664607543423615\\
0.352	3.7309467425875246E-003\\
0.672	2.0480677731937940E-005\\
0.992	5.5597602419665220E-006\\
1.312	3.0038704817126128E-006\\
1.632	1.6643160322893897E-006\\
1.952	9.2341028296301246E-007\\
2.272	5.1246079859134468E-007\\
2.592	2.8441760155961951E-007\\
2.912	1.5785738022412046E-007\\
3.232   8.7615614108397703E-008 \\
3.552    4.8630083515969123E-008 \\
3.872    2.6991987991699210E-008\\
4.192    1.4982032829746363E-008\\
4.512   8.3159592713068318E-009\\
4.832   4.6159324058635613E-009  \\
5.152    2.5621925408166142E-009 \\
5.472    1.4222274981148672E-009\\
5.792    7.8946190259604521E-010 \\
6.112    4.3822565974980020E-010 \\
6.432    2.4325898658791016E-010\\
6.752   1.3503422199462143E-010 \\
7.072    7.4958893317364250E-011\\
7.392   4.1610816535547591E-011\\
7.712   2.3099001012143255E-011 \\
8.032    1.2822833767792353E-011\\
8.352   7.1183599305288166E-012\\
8.672   3.9515772441857144E-012\\
8.992   2.1937558092639210E-012\\
9.312   1.2177884808100889E-012\\
9.632    6.7602580766947727E-013\\
9.952    3.7534156144938105E-013\\
10.272   2.0834882133775065E-013\\
10.592   1.1565794820980969E-013\\
10.912   6.4213930294298707E-014\\
11.232   3.5588212005026057E-014\\
11.552   1.9728924381750375E-014\\
11.872 1.0972785860672496E-014\\
};
\addlegendentry{Mesh 1}

\addplot [color=mycolor1, line width=1.0pt, mark size=2pt, mark=+, mark options={solid, mycolor1}]
  table[row sep=crcr]{%
0.008	 1.2056316640742155\\
0.328	 2.4157581842309886E-003\\
0.648	 5.2958788242123484E-006\\
0.968	 6.8610286923413265E-008\\
1.288	  3.0265919647962541E-008\\
1.608	 1.6178924363290301E-008\\
1.928	 8.6821727382693837E-009\\
2.248	 4.6616583853960164E-009\\
2.568	  2.5032298596356642E-009\\
2.888	 1.3442464134231922E-009\\
3.208    7.2188782642969184E-010\\
3.528    3.8767885876127712E-010\\
3.848    2.0820240407327159E-010\\
4.168    1.1181761574188818E-010\\
4.488    6.0054513803193761E-011\\
4.808   3.2254540789246538E-011\\
5.128    1.7323917144395010E-011\\
5.448   9.3048924139748268E-012\\
5.768   4.9978771880350696E-012\\
6.088   2.6845426094847406E-012\\
6.408    1.4420225601318027E-012\\
6.728   7.7464034719018542E-013\\
7.048   4.1607067947957959E-013\\
7.368   2.2351346835274927E-013\\
7.688    1.2006827269008418E-013\\
8.008   6.4485618198868103E-014\\
8.328   3.4641837348206481E-014\\
8.648    1.8600618275803418E-014\\
8.968   1.0013530935117705E-014\\
9.288   5.3733281705417735E-015\\
9.608    2.8799248300687512E-015\\
9.928   1.5374897684982403E-015\\
10.248  8.1586290292590076E-016\\
10.568  4.5038037677813938E-016\\
10.888  2.5473543972330887E-016\\
11.208  1.2356115884764650E-016\\
11.528 	4.0489009332873070E-017\\
11.848  3.0646217336041485E-017\\
};
\addlegendentry{Mesh 2}

\addplot [color=black, line width=1.0pt, mark size=2pt, mark=triangle, mark options={solid, rotate=270, black}]
  table[row sep=crcr]{%
0.002	 1.3928434505706864  \\
0.322	2.1263876949077819E-003 \\
0.642	3.5833638458823164E-006 \\
0.962	 6.2372559385601008E-009 \\
1.282	  1.2717209246457250E-010\\
1.602	6.1625101616946877E-011 \\
1.922	3.2662042191634389E-011 \\
2.242	1.7339010970026179E-011 \\
2.562	9.2076720926541887E-012 \\
2.882	4.8904860334339786E-012 \\
3.202    2.5978557933083358E-012\\
3.522	  1.3802043552412573E-012\\
3.842    7.3338608532677179E-013\\
4.162	  3.8974303656181845E-013\\
4.482	 2.0715269608797690E-013\\
4.802	1.1010793526848870E-013 \\
5.122	 5.8531090901451927E-014\\
5.442	3.1125797797259366E-014 \\
5.762	 1.6561537090901021E-014\\
6.082	  8.7943681158712887E-015\\
6.402    4.6741008653753493E-015\\
6.722    2.4807009281583124E-015\\
7.042    1.3143644395188082E-015\\
7.362    6.7847593009027515E-016\\
7.682    3.6472133884621968E-016\\
8.002    1.9755553488539218E-016\\
8.322     1.0081276555167215E-016\\
8.642     4.3332666624728670E-017\\
8.962     2.8410590670279439E-017\\
9.282    4.8590046116957233E-018\\
9.602    2.2837118906230337E-018\\
9.922    3.7083432394082094E-018\\
10.242   5.5212991500698123E-018\\
10.562   1.1114689743463785E-017\\
10.882  6.3799439262621517E-018 \\
11.202   6.2980539058024495E-018\\
11.522   4.7842945850185824E-018\\
11.842    2.1304185159507733E-018 \\
};
\addlegendentry{Mesh 3}

\addplot [color=mycolor2, line width=1.0pt, mark size=2pt, mark=o, mark options={solid, mycolor2}]
  table[row sep=crcr]{%
0.0005	 1.4520841577347485\\
0.3205	 2.0551876810986224E-003\\
0.6405	 3.2497227349676488E-006\\
0.9605	5.1775586800299088E-009\\
1.2805	 4.1682230194181094E-011\\
1.6005	1.7309109793154747E-011\\
1.9205	 8.9166605460501605E-012\\
2.2405	4.5991011090166422E-012\\
2.5605	2.3726561345546169E-012\\
2.8805	  1.2255116592442319E-012\\
3.2005	   6.3179391827078264E-013\\
3.5205	   3.2610469683953931E-013\\
3.8405	  1.6835346085991677E-013\\
4.1605	  8.6924187535549194E-014\\
4.4805	  4.4883831079031126E-014\\
4.8005	  2.3181940022062436E-014\\
5.1205	  1.1970690682911867E-014\\
5.4405	  6.1919347440156037E-015\\
5.7605	  3.1911424130306931E-015\\
6.0805	  1.6579193858346314E-015\\
6.4005	   8.4733191435767558E-016\\
6.7205	  4.3243073601168835E-016\\
7.0405     2.1436074672258779E-016\\
7.3605    1.0950309722736079E-016\\
7.6805    5.0352508735290751E-017\\
8.0005    1.9241414682802811E-017\\
8.3205     1.3749560280903553E-017\\
8.6405     1.0573553933034841E-017\\
8.9605    4.0388714228853361E-018\\
9.2805    8.0667306588081035E-018\\
9.6005    4.1163635944894341E-018\\
9.9205    1.6192732662975128E-018\\
10.2405   3.7000918975393345E-018\\
10.5605   4.4206046026219178E-018\\
10.8805   5.7678892399208908E-018\\
11.2005   7.6032655611216526E-018\\
11.5205   1.1521361819072527E-018\\
11.8404   7.3514299081140657E-019\\
};
\addlegendentry{Mesh 4}

\addplot [color=mycolor3, dashed, line width=2.0pt]
  table[row sep=crcr]{%
0.0005	0.989931426512908\\
0.3205	0.00152361965801251\\
0.6405	2.34502794851102e-06\\
0.9605	3.60927089013225e-09\\
1.2805	5.55508788994497e-12\\
1.6005	8.54992667615546e-15\\
1.9205	1.31593320602457e-17\\
2.2405	2.02537433162731e-20\\
2.5605	3.11728677750091e-23\\
2.8805	4.79786709125244e-26\\
};
\addlegendentry{$ t \mapsto e^{-2(\pi^2+1/4)t}$}

\end{axis}
\end{tikzpicture}%
}
      \end{minipage}
     \end{tabular}
     \caption{\textbf{DDFV.} Exponential decay of the relative entropy on a sequence of Cartesian (left) and quadrangular (right) meshes in the degenerate case $\varepsilon=0$ and $\Lambda_{11} = 0.1$.}\label{fig:decayDDFV_epsnull}
    \end{figure}}
    
    \section{Conclusion}
    In this paper, we gave theoretical foundations to the exponential convergence towards equilibrium of finite volume schemes for Drift-Diffusion equations. 
    No-flux boundary conditions as well as Dirichlet conditions at thermal equilibrium are considered. Our approach relies on the discrete entropy method. As in the continuous setting, the long-time behavior of the Fokker-Planck equation is theoretically assessed thanks to functional inequalities. The adaptation to the discrete setting of log-Sobolev and Beckner type inequalities for non-constant reference measures was carried out with this aim. 
    Note that our study encompasses both Two-Point Flux Approximation (TPFA) and Discrete Duality Finite Volumes (DDFV). The later approach is robust with respect to anisotropy and allows for general meshes, but does not lead to monotone discretizations in general. 
    Nonetheless, our method still applies, and could also be extended to other schemes building on Finite Volume methods for anisotropic diffusion (see for instance \cite{Droniou_2014, droniou_2018_gradient}).
    Eventually, we provide numerical evidences of our findings. 
    
    \bigskip
    
    \section*{Acknowledgements} 
    C. Canc\`es, C. Chainais-Hillairet and M. Herda acknowledge support from the Labex CEMPI (ANR-11-LABX-0007-01). C. Chainais-Hillairet also acknowledges the support from Project MoHyCon (ANR-17-CE40-0027-01). The authors also thank Maxime Jonval for his precious help in the implementation of some of the numerical methods.
    \bibliographystyle{siamplain}
    \bibliography{bibli}
    
\end{document}